\documentclass[12pt]{iopart}
\usepackage{amsthm,amsmath}
\usepackage{amssymb,mathrsfs}
\usepackage{graphicx}
\usepackage{epstopdf}
\usepackage{stmaryrd}
\usepackage{bm,cases,multirow,booktabs,enumitem}
\usepackage[colorlinks,linkcolor=blue]{hyperref}

\newtheorem{lemma}{Lemma}[section]

\newtheorem{assumption}{Assumption}[section]
\newtheorem{theorem}{Theorem}[section]
\newtheorem{definition}{Definition}[section]
\eqnobysec  
\graphicspath{{Figures/}{./}}
\allowdisplaybreaks[4]
\begin{document}

\title[Modeling and analysis for TF-KSNS system over bounded domain]{Mathematical modeling and analysis for the chemotactic diffusion in porous media with incompressible Navier-Stokes equations over bounded domain}

\author{Fugui Ma$^1$, Wenyi Tian$^{3,}$\footnote[7]{Author to whom any correspondence should be addressed.} and Weihua Deng$^{1,2}$}

\address{$^1$ School of Mathematics and Statistics, Lanzhou University, Lanzhou 730000, China}
\address{$^2$ State Key Laboratory of Natural Product Chemistry, Lanzhou University, Lanzhou 730000, China}
\address{$^3$ Center for Applied Mathematics and KL-AAGDM, Tianjin University, Tianjin 300072, China}
\ead{mafg17@lzu.edu.cn, twymath@gmail.com and dengwh@lzu.edu.cn}

\vspace{10pt}
\begin{indented}
\item[]November 2024
\end{indented}

\begin{abstract}
Myxobacteria aggregate and generate fruiting bodies in the soil to survive under starvation conditions. Considering soil as a porous medium, the biological mechanism and dynamic behavior of myxobacteria and slime (chemoattractants) affected by favorable environments in the soil can not be well characterized by the classical full parabolic Keller-Segel system combined with the incompressible Navier-Stokes equations. In this work, we employ the continuous time random walk (CTRW) approach to characterize the diffusion behavior of myxobacteria and slime in porous media at the microscale, and develop a new macroscopic model named as the time-fractional Keller-Segel system. Then it is coupled with the incompressible Navier-Stokes equations through transport and buoyancy, resulting in the TF-KSNS system, which reveals the biological mechanism from micro to macro and then appropriately describes the dynamic behavior of the chemotactic diffusion of myxobacteria and slime in the soil. In addition, we demonstrate that the TF-KSNS system associated with initial and no-flux/no-flux/Dirichlet boundary conditions over smoothly bounded domain in $\mathbb{R}^{d}$ ($d\geq2$) admits a local well-posed mild solution, which continuously depends on the initial data with proper regularity under a small initial condition. Moreover, the blow-up of the mild solution is rigorously investigated.
\vspace{1pc}


\end{abstract}

\section{Introduction}
\label{sec:introduction}
In this paper, we consider to discuss the mathematical modeling, local well-posedness analysis, and blow-up of the solution to the following time-fractional fully parabolic Keller-Segel (K-S) system coupled with the incompressible Navier-Stokes (N-S) equations (abbreviated to TF-KSNS)
\begin{subequations}\label{eq:problem}
\begin{align}
  &\partial^{\alpha}_tn+(\mathbf{u}\cdot\nabla)n = \Delta n-\nabla\cdot\left(n\chi(c)\nabla c\right), && x\in\Omega,~ ~ 0<t\le T, \label{eq:problem1a}\\
  &\partial^{\alpha}_tc+(\mathbf{u}\cdot\nabla)c = \Delta c-\gamma c+n,     && x\in\Omega,~ ~ 0<t\le T, \label{eq:problem1b}\\
  &\partial_t\mathbf{u}+(\mathbf{u}\cdot\nabla)\mathbf{u} = \Delta \mathbf{u}-\nabla P+n\nabla\Phi,\quad \nabla\cdot\mathbf{u}=0,&& x\in\Omega,~ ~ 0<t\le T, \label{eq:problem1c}
\end{align}
\end{subequations}
in a given convex, bounded and simply connected domain $\Omega\subseteq\mathbb{R}^d$ ($d\geq2$) with smooth boundary, supplemented with no-flux boundary conditions for $n$ and $c$, a no-slip boundary condition for $\mathbf{u}$,
\begin{equation}\label{eq:BoundaryC}
\frac{\partial n}{\partial\nu}=\frac{\partial c}{\partial\nu}=0,~ ~ {\rm and}~ ~ \mathbf{u}=0, \quad x\in\partial\Omega,~ ~ 0<t\le T,
\end{equation}
and the initial values
\begin{equation}\label{eq:InitialC}
  n(x,0)=n_0(x),~ c(x,0)=c_0(x),~ \mathbf{u}(x,0)=\mathbf{u}_0(x),\quad x\in\Omega,
\end{equation}
where the unknown functions $n=n(x,t)$, $c=c(x,t)$, $\mathbf{u}=[u_1(x,t),\cdots,u_d(x,t)]^{T}$ and $P=P(x,t)$ denote the population density of myxobacteria, the slime concentration, the fluid velocity field, and the pressure, respectively; $\nu$ is the outer normal direction; the parameter function $\chi(c)$ measures the chemotactic sensitivity, which may depend on slime concentration $c$; $\gamma\geq0$ represents the consumption rate of the slime; and $\Phi$ is the gravitational potential function. $\frac{\partial}{\partial \nu}$ denotes the outward normal derivative along the direction $\nu$ on $\partial\Omega$; the operator $\partial^{\alpha}_t$ is the Caputo fractional derivative defined by
\begin{equation}\label{eq:caputo}
\partial^{\alpha}_t\eta(\cdot,t)=\frac{1}{\Gamma(1-\alpha)}\int_{0}^{t}(t-s)^{-\alpha}\partial_s\eta(\cdot,s)ds,\quad t>0,~ ~ \alpha\in(0,1).
\end{equation}
Typically, as $\alpha\rightarrow1$, $\partial^{\alpha}_t$ turns out to be the first-order local differential operator $\partial_t$, and then the problem \eqref{eq:problem} reduces to the classical Keller-Segel system coupled to the Navier-Stokes equations (see, e.g., \cite{Bonilla23,Wang17,Winkler18,Winkler20}). It is indicated in \eqref{eq:problem} that chemotaxis and fluid dynamics are coupled through two components: the transport of myxobacteria and chemoattractants by the fluid, represented by the terms $(\mathbf{u}\cdot\nabla) n$ and $(\mathbf{u}\cdot\nabla) c$; and the external force $n\nabla\Phi$ exerted on the fluid by the myxobacteria due to buoyancy.

\subsection{Background and previous works}
\label{eq:Earlier}
Chemotaxis refers to the directional migration of substances driven by concentration gradients, such phenomenon is widely occurring in nature, such as chemistry \cite{Mandl23}, medicine \cite{Joseph17,Sackmann14}, biology \cite{Anna21,Olson04,Roussos11}, and etc. The mechanism of chemotaxis in biological processes has always been one of the main concerns of experimental scientists \cite{Celani09,Harris12,Tuval05}. Although there are numerous different environmental settings in which bacterial chemotaxis has been observed, the majority of our understanding of this phenomenon comes from laboratory research on model organisms. Subsequently, the field on mathematical modeling of chemotaxis has grown to a wide range of topics, including system modeling, mechanistic basis, and mathematical analysis of the underlying equations.

Traditionally, macroscopic diffusion is defined as the spatiotemporal distribution of a population density of random walkers \cite{Bellouquid16}. At the macroscopic level, Keller and Segel \cite{Keller70}
built the well-known Keller-Segel model
\begin{equation}\label{eq:problem0}
  \left\{
  \begin{aligned}
    &\partial_tn=D_n\Delta n-\nabla\cdot(n\chi(c)\nabla c),    && x\in\Omega,~ ~ 0<t\le T,\\
    &\partial_tc=D_c\Delta c-\gamma c + n,                     && x\in\Omega,~ ~ 0<t\le T,
  \end{aligned}\right.
\end{equation}
where the positive constants $D_n$ and $D_c$ are the diffusivity of cells and chemoattractant, respectively. The term $\chi(c)=1/c$ represents the chemotactic sensitivity of the cells; more specifically, it expresses the tendency of the cells to aggregate due to the difference of concentration in the chemoattractant $c$. This mathematical model successfully describes chemotactic aggregation of cellular slime molds because of its intuitive simplicity, analytical tractability, and capacity to replicate key behaviors of chemotactic populations \cite{Hillen09}. It has become the prevailing model for representing chemotactic dynamics in biological systems on the population level (see, e.g., \cite{Celinski21,Hillen09,KiselevNRY:2023,Painter11,Stevens00}). This model has solutions blowing up for large enough initial conditions in dimensions $d\geq2$, but the solutions are all regular in one dimension, which is confirmed with the patterns in the biological systems (see, e.g., \cite{Calvez08,Winkler10,Winkler18}).

However, the model \eqref{eq:problem0} ignores the interaction between the bacteria, chemoattractants, and the environment from the viewpoint of biology. Actually, the migration of bacteria is substantially affected by changes in the environment. For instance, Tuval et al. \cite{Tuval05} proposed the following chemotaxis-Navier-Stokes model to describe pattern formation in populations of aerobic bacteria interacting with the liquid environment via transport and buoyancy
\begin{equation}\label{eq:KSNS}
  \left\{\begin{aligned}
  &\partial_tn+(\mathbf{u}\cdot\nabla) n=\Delta n-\nabla\cdot(n\chi(c)\nabla c),
     \quad &&x\in\Omega,~ 0<t\le T, \\
  &\partial_tc+(\mathbf{u}\cdot\nabla) c=\Delta c-nf(c),
     \quad &&x\in\Omega,~ 0<t\le T, \\
  &\partial_t\mathbf{u}+(\mathbf{u}\cdot\nabla)\mathbf{u}=\Delta \mathbf{u}+\nabla P+n\nabla\Phi,
     \quad \nabla\cdot\mathbf{u}=0,
     \quad &&x\in\Omega,~ 0<t\le T.
  \end{aligned}\right.
\end{equation}
Due to buoyancy, the bacteria exert force $n\nabla\Phi$ on the fluid, and the source term $n\nabla\Phi$ reflects that the fluctuations in bacterial population density cause forced changes in fluid velocity $\mathbf{u}$ and pressure $P$ with the given gravitational potential $\Phi$. In the chemotactic movement, bacteria migrate to areas with higher concentrations of chemoattractants, both bacteria and chemoattractants are transported by the surrounding fluid.
Wang \cite{Wang17} considered a modified Keller-Segel system coupled with the Navier-Stokes fluid (i.e., the KSNS system) in a bounded domain $\Omega\subset \mathbb{R}^3$, where the terms $\chi(c)$ and $-nf(c)$ are replaced by $\chi(x,n,c)$ and $-c+n$, respectively. Under the condition $|\chi(x,n,c)|\leq C(1+n)^{-\alpha}$ with $\alpha>\frac{1}{3}$, the author proved the existence of global weak solutions of \eref{eq:KSNS}.
Recently, Winkler \cite{Winkler20} took into account the KSNS system with $\chi(c)=1$ and $-nf(c)=-c+n$ in a bounded domain $\Omega\subset \mathbb{R}^{2}$, equipped with the boundary conditions \eqref{eq:BoundaryC} and appropriate initial data. The corresponding initial-boundary value problem was testified to admit a globally defined generalized solution. One can also refer to \cite{Bonilla23,Winkler18} and the references therein for more discussions.

In recent years, the Keller-Segel model with anomalous diffusion has gained popularity in \cite{Azevedo19,Costa23,Escudero06,Langlands2010}, etc. The time-fractional Keller-Segel system \cite{Azevedo19,Costa23} reads:
\begin{equation}\label{eq:problem011}
  \left\{
  \begin{aligned}
    &\partial^{\alpha}_tn=D_n\Delta n-\nabla\cdot(n\chi(c)\nabla c), && x\in\Omega,~ ~ 0<t\le T,\\
    &\partial^{\alpha}_tc=D_c\Delta c-\gamma c + n,                  && x\in\Omega,~ ~ 0<t\le T.
  \end{aligned}\right.
\end{equation}
Azevedo et al. \cite{Azevedo19} proved an existence result of the time-fractional Keller-Segel system \eqref{eq:problem011} with small initial data in a class of Besov-Morrey spaces in $\mathbb{R}^d$, $d\geq2$. In \cite{Costa23}, Costa and the co-authors dealt with local existence and blow-up of the solution to the time-fractional Keller-Segel system \eref{eq:problem011} in the setting of Lebesgue and Besov spaces for chemotaxis under homogeneous Neumann boundary conditions in a smooth domain of $\mathbb{R}^d$, $d\geq2$. Moreover, the Keller-Segel system \eref{eq:problem011} was extended by Escudero in \cite{Escudero06} to the space fractional Keller-Segel system, where the dispersal is characterized by the fractional Laplacian $(-\Delta)^{\alpha/2}$ with $\alpha\in(0,2)$. Li et al. \cite{Li2018} further investigated the Cauchy problems for the nonlinear fractional time-space generalized Keller-Segel system. Besides that, the time-space fractional Keller-Segel system was coupled with the incompressible Navier-Stokes equations by Jiang et al. \cite{Jiang22,Jiang24b,Jiang24a}, where the corresponding Cauchy initial problem was investigated, including the well-posedness, time decay, as well as asymptotic stability of its mild solution in the Lebesgue or Besov-Morrey spaces in $\mathbb{R}^d$, $d\geq2$.

The time-fractional Keller-Segel system \eqref{eq:problem011} in most existing literatures was simply obtained with the first-order derivative in \eqref{eq:problem0} replaced by a time-fractional analogue. However, the physical and biological explanation of \eqref{eq:problem011} describing the chemotactic diffusion combining anomalous diffusion still remains insufficient despite the efforts made in \cite{Langlands2010} and \cite{Bellouquid16}, let alone its coupled system with incompressible Navier-Stokes equations. To the best of our knowledge, anomalous diffusion models in the previous works were considered in static medium, such as time-fractional Fokker-Planck equations \cite{Klafter11}, time-fractional forward or backward Feynman-Kac equations \cite{CarmiTB:2010,DengHWX:2020}, time-fractional mobile-immobile equations \cite{Ma23}, and etc.
Nevertheless, many more biological processes, physical processes, and chemical reactions are occurring in the non-static medium; at the same time, the corresponding phenomena bring new inspirations and challenges in the modeling, analysis and simulations. Obviously, in our model \eqref{eq:problem}, a non-static liquid field is described by the incompressible Navier-Stokes equations, which is the environment involving the chemotactic diffusion in the porous media.
Moreover, the time-fractional Caputo derivative $\partial^{\alpha}_t$ in \eqref{eq:problem1a}-\eqref{eq:problem1b} is an integral-differential and non-local operator with historical memory. Therefore, the techniques for analyzing traditional Keller-Segel systems with the first-order time derivative, as well as the existing existence and regularity results, can not be directly generalized to the non-local analogue \eqref{eq:problem1a}-\eqref{eq:problem1b}, which poses significant challenges for analysis. Another challenge comes from the handling of bilinear terms $(\mathbf{u}\cdot\nabla)n$ and $(\mathbf{u}\cdot\nabla)c$ in \eqref{eq:problem}. Whenever the time-fractional Keller-Segel system is coupled with the incompressible Navier-Stokes equations, the analytical challenge would be even bigger for the investigations of well-posedness and blow-up of the solution. Thus, it is necessary to fill this gap for the TF-KSNS system \eqref{eq:problem}.

\subsection{Major contributions}
Taking into account the distinctive advantages of the aforementioned works, the main contributions of the present work consist of two aspects.

The first one is to mathematically derive the model \eqref{eq:problem} for describing the chemotactic diffusion with anomalous diffusion in porous media; it provides reasonable mathematical and physical explanations of the biological process governed by the model.
Specifically, the chemotactic diffusion of myxobacteria from a microscopic perspective in porous media can be characterized by the stochastic process with a power-law distribution of waiting time based on the continuous time random walk (CTRW), then the time-fractional Keller-Segel system is naturally derived into the mathematical model. In addition, at the microlevel, the sliding and diffusion of myxobacteria and slime (chemoattractants) in the soil are influenced by the surrounding environment in specific biological processes, such as liquid flow fields and rainwater present in the soil. Similar to the modeling of \eqref{eq:KSNS}, we will also strongly couple the derived time-fractional Keller-Segel system with the incompressible Navier-Stokes equations describing the liquid flow fields and thus obtain the TF-KSNS system \eqref{eq:problem}.

The other one is to perform a complete local well-posedness analysis of the TF-KSNS system \eqref{eq:problem} and to address the blow-up of its mild solution over a bounded domain $\Omega\subset \mathbb{R}^d$ ($d\geq2$) with smooth boundary. To begin with, a complete metric space is first constructed, in which a map is simultaneously designed according to the expression of the mild solution, and it is verified that the map satisfies the conditions of the Banach fixed point theorem. Furthermore, the continuous dependence of the mild solution on the initial values is proved, and the asymptotic property of the mild solution is also discussed. We also prove the uniqueness of the continuous extension of the mild solution. Meanwhile, the blow-up of the solution is also analyzed by using the method of contradiction.

\subsection{Outline of the paper}
The paper is arranged as follows. In Section \ref{sec:Mathmodeling}, we build the mathematical modeling of the TF-KSNS system \eqref{eq:problem} from microscopic to macroscopic under certain fundamental assumptions. Section \ref{sec:Mathanalyses} presents the key mathematical analysis results, including the local well-posedness and finite time blow-up of the mild solution to the TF-KSNS system \eqref{eq:problem}-\eqref{eq:InitialC}. The detailed analyses of the two results are provided in Sections \ref{sec:estmatMS} and \ref{sec:blow-up}, respectively. Some conclusions are finally drawn in Section \ref{sec:Conclusions}.

\section{Mathematical modeling: From microscopic to macroscopic}
\label{sec:Mathmodeling}
\subsection{The biology}
It is ubiquitous in environments that myxobacteria aggregate and eventually form fruiting bodies under starvation conditions (see, e.g., \cite{Dworkin93,Stevens00}), new slimes are produced to help myxobacteria survive. Sliding myxobacteria employ chemotaxis to find chemical hotspots, which are high concentration areas of myxobacteria  \cite{Brumley19}, we can refer to Fig. \ref{fig:shiyitu} for an intuitive understanding of the motility and directivity of myxobacteria. In order to get a deeper grasp of the biological mechanisms of myxobacteria chemotaxis, Steven \cite{Stevens00} constructed a stochastic cellular automaton, in which the author took into account the biological hypotheses that myxobacteria produce slime traces where they like to glide on.

It is mentioned in \cite{Anna21} that ``natural soils are host to a high density and diversity of microorganisms, and even deep-earth porous rocks provide a habitat for active microbial communities." The transport and distribution of myxobacteria in the soil (porous media) remain unclear due to their disordered flow and associated chemical gradients. In the next subsection, we consider to provide a mathematical explanation for a specific biological scenario depicted in Fig. \ref{fig:shiyitu} that myxobacteria aggregate and form fruiting bodies to secrete large amounts of slime for survival in the soil and derive the corresponding mathematical model.
\begin{figure}[!ht]%
  \centering
  \includegraphics[height=4.23cm,width=9.20cm]{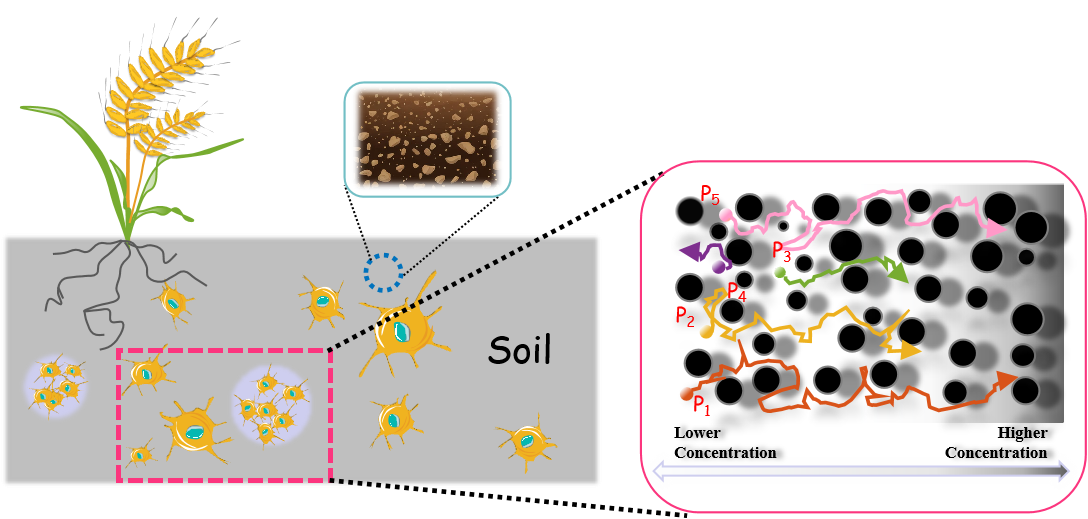}
  \caption{Illustration of the biological environment and the chemotactic diffusion: the myxobacteria (i.e., $\{P_j\}_{j\geq1, j\in\mathbb{N}_0}$) glide in a porous medium towards a relatively high concentration of chemical signals, where black round holes with different sizes refer to obstacles in porous media, and the shades of gray in the background indicate the concentration of the chemoattractant; the darker the color, the higher the concentration. Myxobacteria aggregate and form fruiting bodies that produce new slime for survival.}
  \label{fig:shiyitu}
\end{figure}

\subsection{Time-fractional chemotaxis diffusion equations with space- and time-dependent forces}
Based on the biological mechanisms revealed in the previous subsection, we will employ the theory of stochastic processes at the microscale and the continuous time random walk (CTRW) approach to derive the time-fractional Keller-Segel system that characterizes the chemotactic diffusion of myxobacteria and slime in the soil.

In \cite{Langlands2010}, Langlands and Henry utilized the balance equation \eqref{eq:balanced} to describe the chemotaxis and diffusion dynamics of myxobacteria, where a special transition probability density function was constructed based on slime concentration.
To further uncover the biological mechanism, we adopt the transfer probability density function proposed by Steven \cite{Stevens00}, to derive a micro dynamic model applicable to the biological scenario specified in the previous subsection and derive the corresponding macroscopic governing equations by the Laplace transform. The detailed modeling and derivation process are as follows.

To model the chemotactic diffusion of myxobacteria, we denote the probability distribution $n(x,t)$ to represent the concentration of myxobacteria at time $t$ and location $x$.
By using the generalized CTRW balance equation in \cite{Henry10,Langlands2010},
it incorporates the intricate mechanisms of chemotaxis as
\begin{align}\label{eq:balanced}
  n(x,t)=n(x,0)\Psi(t)+\int_{0}^{t}&\big[p_r(x-\delta x\rightarrow x,s)n(x-\delta x,s) \notag \\
                                  &+p_{l}(x+\delta x\rightarrow x,s)n(x+\delta x,s)\big]\psi(t-s)ds,
\end{align}
where $\psi(t)$ is the probability density function (PDF) of waiting time, and $\Psi(t)$ is called survival probability, i.e., the waiting time on a site exceeds $t$, defined by $\Psi(t):=\int_{t}^{\infty}\psi(s)ds=1-\int_{0}^{t}\psi(s)ds$. In \eqref{eq:balanced}, $p_r(x-\delta x\rightarrow x,t)$ and $p_l(x+\delta x\rightarrow x,t)$ are the transition probabilities of jumping from the adjacent grid points ($x\pm\delta x$) in the right and left directions to $x$, respectively. When different distribution functions $\psi(t)$, $p_r(x,t)$, and $p_l(x,t)$ are selected in \eqref{eq:balanced}, we will observe different chemotaxis and diffusion dynamics of myxobacteria. Therefore, it is crucial to pick appropriate distribution functions $\psi(t)$, $p_r(x,t)$ and $p_l(x,t)$ for accurately describing the chemotaxis and diffusion process of myxobacteria.

As is well known, the sub-diffusion process has an advantage for describing the diffusion-transport process of particles in porous media. Inspired by this and taking into account the sliding trace of myxobacteria and the hindrance of narrow gaps in the soil to myxobacteria migration, we take a power law PDF of waiting time as
\begin{equation}\label{eq:psi}
  \psi(t)\sim \tau^\alpha/t^{1+\alpha}, \quad 0<\alpha<1,
\end{equation}
for $t\gg\tau$, where $\tau$ is a characteristic waiting time scale \cite{Henry10,Langlands2010}.
Taking inspiration from the work of Stevens \cite[(8) and (9)]{Stevens00}, the transition probability law governing jumps to the left or right direction in \eqref{eq:balanced} is determined by the relative concentration of the slime (chemoattractant) on either side of the current location, as follows
\begin{equation}\label{eq:tp}
  \left\{\begin{aligned}
  p_{l}(x\rightarrow x-\delta x,t)&=\frac{v(x-\delta x,t)}{v(x-\delta x,t)+v(x+\delta x,t)},\\
  p_{r}(x\rightarrow x+\delta x,t)&=\frac{v(x+\delta x,t)}{v(x-\delta x,t)+v(x+\delta x,t)},
  \end{aligned}\right.
\end{equation}
where $v:=g(c)$ depends on the density (concentration) of the slime (chemoattractant) at time $t$ and point $x$, $g$ is a given function. The transition probability in \eqref{eq:tp} indicates that myxobacteria are attracted by chemoattractants from different directions, which is reasonable in the biological mechanism.

Next, we are ready to derive the governing equations for modeling the chemotaxis and diffusion process of myxobacteria and slime in the soil.
Employing the Laplace transform to the balance equation \eqref{eq:balanced} and the relationship $\hat{\Psi}(z)=z^{-1}(1-\hat{\psi}(z))$, we can obtain the discrete space evolution equation
\begin{equation}\label{eq:LapMW}
\begin{aligned}
z\hat{n}(x,z)-\hat{n}(x,0)
=\frac{\hat{\psi}(z)}{\hat{\Psi}(z)}\big\{-\hat{n}(x,z)
  &+\mathcal{L}_{t\rightarrow z}\{p_r(x-\delta x\rightarrow x,t)n(x-\delta x,t)\}(z)\cr
  &+\mathcal{L}_{t\rightarrow z}\{p_l(x+\delta x\rightarrow x,t)n(x+\delta x,t)\}(z)\big\},
\end{aligned}\end{equation}
where the notation ``$\hat{w}(z)$'' represents the Laplace transform of the function $v(t)$ with respect to the time variable $t$, which is also denoted as $\hat{w}(z):=\mathcal{L}_{t\rightarrow z}\{w(t)\}(z)$. According to \eqref{eq:psi} and $\hat{\Psi}(z)=z^{-1}(1-\hat{\psi}(z))$, it obtains that
\begin{equation}\label{eq:ratio}
\frac{\hat{\psi}(z)}{\hat{\Psi}(z)}\sim D_\alpha\frac{z^{1-\alpha}}{\tau^\alpha}, \quad {\rm with}\quad
D_\alpha=\frac{\alpha}{\Gamma(1-\alpha)}.
\end{equation}
Then it derives from \eqref{eq:ratio} and the inverse Laplace transform of \eqref{eq:LapMW} that
\begin{align}\label{eq:tdMW}
  \frac{dn(x,t)}{dt}
  =\frac{D_\alpha}{\tau^{\alpha}}\frac{d^{\alpha-1}}{dt^{\alpha-1}}
    \big\{-n(x,t)&+p_r(x-\delta x\rightarrow x,t)n(x-\delta x,t)  \notag \\
    &+p_l(x+\delta x\rightarrow x,t)n(x+\delta x,t)\big\},
\end{align}
where it employs the property of the Laplace transform of the Riemann-Liouville fractional derivative ($
\mathcal{L}_{t\rightarrow z}\big\{\frac{d^\alpha}{dt^{\alpha}}w(t)\big\}(z)=z^{\alpha}\hat{w}(z)
-\frac{d^{\alpha-1}}{dt^{\alpha-1}}w(t)\big|_{t=0},~ 0<\alpha<1
$)
and the vanishing behavior near the lower terminal ($\frac{d^{\alpha-1}}{dt^{\alpha-1}}w(t)|_{t=0}=0$, \cite[P{100}-P{105}]{Podlubny99}).

Subsequently, we consider the continuous limit of the aforementioned spatially discrete formulation \eqref{eq:tdMW} with $\delta x\rightarrow0$. To this end, we utilize \eqref{eq:tp} and the Taylor expansions of the lattice functions at the point $x$ up to and including terms of order $(\delta x)^2$. It yields from \eqref{eq:tp} the following expression:
\begin{align*}
  &-n(x,t)+p_r(x-\delta x\rightarrow x,t)n(x-\delta x,t)
           +p_l(x+\delta x\rightarrow x,t)n(x+\delta x,t)\\
  &=\frac{c(x,t)}{c(x,t)+c(x-2\delta x)}n(x-\delta x,t)
           +\frac{c(x,t)}{c(x,t)+c(x+2\delta x)}n(x+\delta x,t)-n(x,t) \\
  &=\frac{v(x,t)\big[\big(v(x,t)+v(x+2\delta x,t)\big)n(x-\delta x,t)
           +\big(v(x,t)+v(x-2\delta x,t)\big)n(x+\delta x,t)\big]}
            {\big(v(x,t)+v(x-2\delta x,t)\big)\big(v(x,t)+v(x+2\delta x,t)\big)} \\
  &~~~~ -\frac{n(x,t)\big(v(x,t)+v(x-2\delta x,t)\big)\big(v(x,t)+v(x+2\delta x,t)\big)}
            {\big(v(x,t)+v(x-2\delta x,t)\big)\big(v(x,t)+v(x+2\delta x,t)\big)}.
\end{align*}
In addition, with simple calculations, it holds that
\begin{align*}
  \frac{1}{\big(v(x,t)+v(x-2\delta_x,t)\big)\big(v(x,t)+v(x+2\delta x,t)\big)}
  \approx\frac{1}{4v^2(x,t)},\quad {\rm as}~ ~ \delta x\rightarrow0,
\end{align*}
\begin{align*}
  &v(x,t)\big[\big(v(x,t)+v(x+2\delta x,t)\big)n(x-\delta x,t)
              +\big(v(x,t)+v(x-2\delta x,t)\big)n(x+\delta x,t)\big]  \\
  &\approx 4v(x,t)\Big(v(x,t)n(x,t)+\frac{1}{2}(\delta x)^2v(x,t)\partial_{xx}n(x,t)   \\
  &\hskip2.2cm-(\delta x)^2\partial_xv(x,t)\partial_xn(x,t)+(\delta x)^2n(x,t)\partial_{xx}v(x,t)\Big),
\end{align*}
and
\begin{align*}
 &n(x,t)\big(v(x,t)+v(x-2\delta x,t)\big)\big(v(x,t)+v(x+2\delta x,t)\big) \\
 &\approx 4n(x,t)(v(x,t))^2+8n(x,t)(\delta x)^2v(x,t)\partial_{xx}v(x,t)
         -4(\delta x)^2n(x,t)\big(\partial_{x}v(x,t)\big)^{2}.
\end{align*}
Then, the above results deduce that
\begin{equation}\label{eq:tdMW2}
  \frac{dn(x,t)}{dt}
  =\frac{d^{\alpha-1}}{dt^{\alpha-1}}
  \Big\{\mathcal{D}_{\alpha}\partial_{xx}n(x,t)-\mathcal{T}_{\alpha}\partial_{x}\Big(\frac{n(x,t)}{v(x,t)}\partial _{x}v(x,t)\Big)\Big\},
\end{equation}
where $\mathcal{D}_{\alpha}=\frac{D_\alpha (\delta x)^2}{2\tau^{\alpha}}$ and $\mathcal{T}_{\alpha}=\frac{D_\alpha (\delta x)^2}{\tau^{\alpha}}$ as $\delta x,\tau\rightarrow0$.
Thus, by using $v=g(c)$ and the connection between Riemann-Liouville and Caputo fractional derivatives (see \cite{Podlubny99}), it leads to
\begin{equation}\label{eq:tdMW4}
  \partial_t^{\alpha}n(x,t)
  =\mathcal{D}_{\alpha}\partial_{xx}n(x,t)-\mathcal{T}_{\alpha}\partial_{x}\big(n(x,t)\chi(c)\partial _{x}v(x,t)\big),
\end{equation}
where $\chi(c)=g'(c)/g(c)$. For $v(x,t)=c(x,t)$ in (\ref{eq:tp}), it refers to the model with $\chi(c)=1/c$ in \cite{Bellomo2015,Hillen09}; for $v(x,t)=e^{\beta c(x,t)}$, it covers the model with $\chi(c)\equiv\beta$ in \cite{Bellomo2015,Langlands2010}. Upon selecting suitable coefficients, we obtain the equation (\ref{eq:problem1a}) including the transport term $(\mathbf{u}\cdot\nabla) n$ by virtue of the fluid velocity. Obviously, as $\alpha\rightarrow1^{+}$, (\ref{eq:tdMW4}) reduces to the first equation in the classical Keller-Segel system (\ref{eq:problem0}).

As we know, the diffusion of particles (slime molecules) in porous media (soil) can be precisely described by the subdiffusion equation \cite{Metzler00}, i.e.,
$\partial_{t}^{\alpha}c(x,t)=\Delta c(x,t)$. Meanwhile, the slime in the soil will continuously dry up and deplete over time, and the aggregation of myxobacteria can create new slime. As a result, the reaction and source term $-\gamma c(x,t)+n(x,t)$, characterizing the decay of slime and new generation of slime, are merged into the subdiffusion equation together with the transport term $(\mathbf{u}\cdot\nabla) c$. Then it produces \eqref{eq:problem1b} to explain the diffusion mechanism of chemoattractants (slime) in the soil (porous media).

Analogous to the derivation of the system \eqref{eq:KSNS} from \eqref{eq:problem0}, we also strongly couple \eqref{eq:problem1a} and \eqref{eq:problem1b} with the well-known incompressible Navier-Stokes equations \eqref{eq:problem1c} to accurately characterize the biological processes and the living environment, where the migration and chemotaxis of myxobacteria and slime under space- and time-dependent forces are driven by the fluid flow in the soil.

\section{Main results}\label{sec:Mathanalyses}
This section presents the analysis results for the TF-KSNS system \eqref{eq:problem} with the no-flux/no-flux/Dirichlet boundary conditions \eqref{eq:BoundaryC} in smoothly bounded domain and the initial data in \eqref{eq:InitialC}. Under some assumptions on the initial values, the problem admits a unique local mild solution; the blow-up and asymptotic behaviors of the mild solution are also established.
\subsection{Preliminaries}
\label{sec:preliminary}
We denote $L^{p}(\Omega)$ ($p\ge1$) as the standard Lebesgue space and $W^{m,p}(\Omega)$ the Sobolev space.
In order to analyze the well-posedness and regularity of \eqref{eq:problem}-\eqref{eq:InitialC} as precisely as possible, we need some $L^p$-$L^q$ estimates of the heat semigroup $e^{t\Delta}$ under the homogenous Neumann and Dirichlet boundary conditions in Lemmas \ref{lem:Heatgroup} and \ref{lem:Dirichletgroup}, respectively.
\begin{lemma}[\cite{Celinski21,Winkler10}]\label{lem:Heatgroup}
  Let $(e^{t\Delta})_{t\geq0}$ be the heat semigroup under the homogenous Neumann boundary condition.
  Then there exist some positive constants $C_i$ depending on $\Omega$, $p$ and $q$ such that the following estimates hold.
  \begin{enumerate}[label=(\roman*)]
    \item If $2\leq p<\infty$, for all $t>0$ and $\omega\in W^{1,p}(\Omega)$, then
          \begin{equation}\label{eq:HG1}
            \big\|\nabla e^{t\Delta}\omega\big\|_{L^{p}(\Omega)}\leq C_1 \|\nabla \omega\|_{L^{p}(\Omega)}.
          \end{equation}
    \item If $1\leq p\leq q\leq\infty$, for all $t>0$ and $\omega\in L^{p}(\Omega)$, then
          \begin{equation}\label{eq:HG2}
            \big\|e^{t\Delta}\omega\big\|_{L^{q}(\Omega)}\leq C_2 \big(1+t^{-\frac{d}{2}(\frac{1}{p}-\frac{1}{q})}\big)\|\omega\|_{L^{p}(\Omega)}.
          \end{equation}
    \item If $1\leq p\leq q\leq\infty$, for all $t>0$ and $\omega\in L^{p}(\Omega)$, then
          \begin{equation}\label{eq:HG3}
            \big\|\nabla e^{t\Delta}\omega\big\|_{L^{q}(\Omega)}\leq C_3 \big(1+t^{-\frac{1}{2}-\frac{d}{2}(\frac{1}{p}-\frac{1}{q})}\big)\|\omega\|_{L^{p}(\Omega)}.
          \end{equation}
    \item If $1< p\leq q\leq\infty$, for all $t>0$ and $\omega\in L^{p}(\Omega;\mathbb{R}^{d})$, then
          \begin{equation}\label{eq:HG4}
            \big\|e^{t\Delta}\nabla\cdot \omega\big\|_{L^{q}(\Omega)}\leq C_4 \big(1+t^{-\frac{1}{2}-\frac{d}{2}(\frac{1}{p}-\frac{1}{q})}\big)\|\omega\|_{L^{p}(\Omega)}.
          \end{equation}
  \end{enumerate}
\end{lemma}

For the heat semigroup $(e^{t\Delta})_{t\geq0}$ with the homogenous Dirichlet boundary condition, it also has the analogous results to the Neumann boundary case.
\begin{lemma}\label{lem:Dirichletgroup}
  Let $\Omega\in \mathbb{R}^{d}$ ($d\geq2$) be a domain of class $C^{1,1}$, and $\{e^{t\Delta}\}_{t\geq0}$ the Dirichlet heat semigroup in $\Omega$.
  \begin{enumerate}[label=(\roman*)]
    \item For $1\leq p\leq q\leq\infty$ and all $\omega\in L^{p}(\Omega)$, it holds
      \begin{equation}\label{eq:semigroupLap}
         \|e^{t\Delta}\omega\|_{L^{q}(\Omega)}\leq C_{1}t^{-\frac{d}{2}(\frac{1}{p}-\frac{1}{q})}\|\omega\|_{L^{p}(\Omega)},~~ \forall~ t>0.
      \end{equation}
    \item For $1\leq p\leq q\leq\infty$ and all $\omega\in L^{p}(\Omega)$, it holds
        \begin{equation}\label{eq:GradsemigroupLap}
          \|\nabla e^{t\Delta}\omega\|_{L^{q}(\Omega)}\leq C_{2} \big(1+t^{-\frac{1}{2}-\frac{d}{2}(\frac{1}{p}-\frac{1}{q})}\big)
                 \|\omega\|_{L^{p}(\Omega)},~~ \forall~ t>0.
        \end{equation}
    \item For $1\leq p\leq q\leq\infty$ and all $\omega\in L^{p}(\Omega;\mathbb{R}^d)$, it holds
        \begin{equation}\label{eq:DivsemigroupLap}
          \| e^{t\Delta}\nabla\cdot\omega\|_{L^{q}(\Omega)}
                 \leq C_{3} \big(1+t^{-\frac{1}{2}-\frac{d}{2}(\frac{1}{p}-\frac{1}{q})}\big)
                 \|\omega\|_{L^{p}(\Omega)},~~ \forall~ t>0.
        \end{equation}
  \end{enumerate}
\end{lemma}
\begin{proof}
  Regarding (i), the estimate \eqref{eq:semigroupLap} is a better-known result in \cite[Proposition 48.4]{Quittner07}. As for (ii), \eqref{eq:GradsemigroupLap} can be derived by the heat kernel estimate in \cite{KimS06} and Young's inequality for convolution.

  Now we turn to prove (iii). By \cite[Corollary 1.4]{Brezis10}, there holds
  \begin{equation}\label{eq:Corol}
    \big\|e^{t\Delta}\nabla\cdot\omega\big\|_{L^{q}(\Omega)}
    =\sup_{\substack{\phi\in L^{q^*}(\Omega)\\ \|\phi\|_{L^{q^*}(\Omega)}\leq1}}
    \Big|\int_{\Omega}e^{t\Delta}\nabla\cdot\omega\phi\Big|
    =\max_{\substack{\phi\in L^{q^*}(\Omega)\\ \|\phi\|_{L^{q^*}(\Omega)}\leq1}}
    \Big|\int_{\Omega}e^{t\Delta}\nabla\cdot\omega \phi\Big|.
  \end{equation}
  Note that $e^{t\Delta}$ is a self-adjoint operator. Then, by an analogous approach in the proof of Lemma 1.3 (iv) in \cite{Winkler10}, for any $\phi\in C_{0}^{\infty}(\Omega)$, integrating by parts and employing (\ref{eq:GradsemigroupLap}) with $1\leq q\leq p<\infty$ obtains that
  \begin{align}\label{eq:etmCorol}
    \Big|\int_{\Omega}e^{t\Delta}\nabla\cdot\omega \phi\Big|
    &\leq\Big|-\int_{\Omega}\omega\cdot\nabla e^{t\Delta}\phi\Big|
     \leq \|\omega\|_{L^{p}}\big\|\nabla e^{t\Delta}\phi\big\|_{L^{p*}}        \notag\\
    &\leq C\big(1+t^{-\frac{1}{2}-\frac{d}{2}(\frac{1}{q^*}-\frac{1}{p*})}\big)
    \|\omega\|_{L^{p}(\Omega)}\|\phi\|_{L^{q^*}(\Omega)},
  \end{align}
  where $\frac{1}{p}+\frac{1}{p*}=1$, $\frac{1}{q}+\frac{1}{q*}=1$, and $\frac{1}{q^*}-\frac{1}{p*}=\frac{1}{p}-\frac{1}{q}$. Since $C_{0}^{\infty}(\Omega)$ is dense in $L^{p}(\Omega)$ for any $1\leq p<+\infty$ \cite[Corollary 4.23]{Brezis10}, then we derive from (\ref{eq:Corol}) and (\ref{eq:etmCorol}) that the
  estimate (\ref{eq:DivsemigroupLap}) holds for $1\leq p\leq q<\infty$.
  Furthermore, for the case $q=p=\infty$, by choosing $\omega\in (C_{0}^{\infty})^{d}$
  as in \cite[Lemma 2.1]{Cao15} and taking $1\le\tilde{q}<p=\infty$,
  we have from \eqref{eq:semigroupLap} that
  \begin{align*}
  \big\|e^{t\Delta}\nabla\cdot\omega\big\|_{L^{\infty}(\Omega)}
    &=\big\|e^{\frac{t}{2}\Delta}\big(e^{\frac{t}{2}\Delta}\nabla\cdot\omega\big)\big\|_{L^{\infty}(\Omega)}
    \leq C\big(1+t^{-\frac{1}{2}-\frac{d}{2\tilde{q}}}\big)\|\omega\|_{L^{\tilde{q}}(\Omega)}.
  \end{align*}
  As $\Omega$ being a smooth and bounded domain, therefore, taking the limit $\tilde{q}\rightarrow\infty$, we can obtain $\|e^{t\Delta}\nabla\cdot\omega\|_{L^{\infty}(\Omega)}\leq c(1+t^{-\frac{1}{2}})\|\omega\|_{L^{\infty}(\Omega)}$. The proof is completed.
\end{proof}

Let $\kappa>-1$, $\lambda\in \mathbb{C}$, the Wright function \cite{Podlubny99} is defined by the complex convergent series
$
  W_{\kappa,\lambda}(z):=\sum_{j=0}^{\infty}\frac{z^{j}}{j!\Gamma(\kappa j+\lambda)},~ z\in\mathbb{C},
$
where $\Gamma(\cdot)$ is the Euler-Gamma function.
The Mainardi function denoted as $M_{\alpha}(z)$ is a specific case of the Wright function and given by $M_{\alpha}(z):=W_{-\alpha,1-\alpha}(-z)$. Following \cite{Azevedo19}, it has been proved that $M_{\alpha}$ is non-negative and satisfies
\begin{equation}\label{eq:Mainardi}
  \int_{0}^{\infty}t^{\gamma}M_{\alpha}(t)dt=\frac{\Gamma(1+\gamma)}{\Gamma(1+\alpha\gamma)},\quad~ \gamma>-1.
\end{equation}
The third one is the Mittag-Leffler function defined by
$
  E_{\alpha,\beta}(z):=\sum_{j=0}^{\infty}\frac{z^{j}}{\Gamma(\alpha j+\beta)}
$,
and
$ E_{\alpha}(z):=E_{\alpha,1}(z),~ z\in\mathbb{C}.
$
The Mainardi function $M_{\alpha}(z)$ and the Mittag-Leffler function $E_{\alpha,\beta}(z)$ satisfy the following properties (see, e.g., \cite{Costa23,Andrade22})
\begin{equation}\label{eq:EDelta}
  \left\{
  \begin{aligned}
    &E_{\alpha}(t^{\alpha}\Delta)=\int_{0}^{\infty}M_{\alpha}(s)e^{st^{\alpha}\Delta}ds, \\
    &E_{\alpha,\alpha}(t^{\alpha}\Delta)=\int_{0}^{\infty}\alpha s M_{\alpha}(s)e^{st^{\alpha}\Delta}ds,
  \end{aligned}\right.
\end{equation}
\begin{equation}\label{eq:EDelta2}
  \left\{\begin{aligned}
    &E_{\alpha}\left(t^{\alpha}(\Delta-\gamma)\right)
      =\int_{0}^{\infty}M_{\alpha}(s)H_{\gamma}(st^{\alpha})ds, \\
    &E_{\alpha,\alpha}\left(t^{\alpha}(\Delta-\gamma)\right)
      =\int_{0}^{\infty}\alpha s M_{\alpha}(s)H_{\gamma}(st^{\alpha})ds,
  \end{aligned}\right.
\end{equation}
where $H_{\gamma}(t):=e^{t(\Delta-\gamma)}: L^{p}(\Omega)\rightarrow L^{q}(\Omega)$ is the damped heat semigroup with $\gamma>0$.

The Mittag-Leffler functions involving the Laplacian, called Mittag-Leffler operators, play a central role to represent the mild solution in \eqref{eq:DuhamelP} and satisfy some important properties in the following lemmas.

\begin{lemma}[Strong continuity \cite{Andrade22,Costa23}]\label{lem:continuous}
  The families $(E_{\alpha}(t^{\alpha}\Delta))_{t\geq0}$, $(E_{\alpha}(t^{\alpha}(\Delta-\gamma)))_{t\geq0}$,
  $(E_{\alpha,\alpha}(t^{\alpha}\Delta))_{t\geq0}$, and $(E_{\alpha,\alpha}(t^{\alpha}(\Delta-\gamma)))_{t\geq0}$ are strongly continuous with respect to variable $t$ in $L^{q}(\Omega)$ with $1\leq q<\infty$.
\end{lemma}

Based on the $L^p$-$L^q$ estimates for the Neumann heat semigroup in Lemma \ref{lem:Heatgroup}, we have the following estimates of the Mittag-Leffler operators.
\begin{lemma}\label{lem:MLbounds}
  Let $E_{\alpha}(t^{\alpha}\Delta)$, $E_{\alpha}(t^{\alpha}(\Delta-\gamma))$, $E_{\alpha,\alpha}(t^{\alpha}\Delta)$ and $E_{\alpha,\alpha}(t^{\alpha}(\Delta-\gamma))$ be defined by \eqref{eq:EDelta}-\eqref{eq:EDelta2}. For $t>0$, $\rho>2$ and $1\leq q\leq\infty$, there exists a constant $C>0$ only depending on $\Omega$, $\alpha$, $d$ and $q$, such that the following estimates hold:
  \begin{align}
    &\big\|E_{\alpha}(t^{\alpha}\Delta)\big\|_{L^{q}(\Omega)\rightarrow L^{q}(\Omega)}\leq C,\label{eq:Ea}\\
    &\big\|E_{\alpha,\alpha}(t^{\alpha}(\Delta-\gamma))\big\|_{L^{q}(\Omega)
          \rightarrow L^{q}(\Omega)}\leq C,\label{eq:Ee}\\
    &\big\| E_{\alpha,\alpha}(t^{\alpha}(\Delta-\gamma))\big\|_{L^{\frac{\rho q}{\rho+1}}(\Omega)
          \rightarrow L^{\rho q}(\Omega)}\leq C(1+t^{-\frac{\alpha d}{2 q}}),~ ~ q>d,\label{eq:Ed}\\
    &\big\|\nabla E_{\alpha,\alpha}(t^{\alpha}(\Delta-\gamma))\big\|_{L^{q}(\Omega)
          \rightarrow L^{q}(\Omega)}\leq C (1+t^{-\frac{\alpha}{2}}),\label{eq:Ef}\\
    &\big\|\nabla E_{\alpha,\alpha}(t^{\alpha}(\Delta-\gamma))\big\|_{L^{\frac{\rho q}{\rho+1}}(\Omega)
          \rightarrow L^{\rho q}(\Omega)}
         \leq C(1+t^{-\frac{\alpha}{2}-\frac{\alpha d}{2q}}),~ ~ q>d,\label{eq:Ec}\\
    &\big\|E_{\alpha,\alpha}(t^{\alpha}\Delta)\nabla\cdot\big\|_{L^{\rho q}(\Omega)
          \rightarrow L^{\rho q}(\Omega)}\leq C(1+t^{-\frac{\alpha}{2}})\label{eq:Eg},\\
    &\big\|E_{\alpha,\alpha}(t^{\alpha}\Delta)\nabla\cdot\big\|_{L^{\frac{\rho q}{\rho+1}}(\Omega)
          \rightarrow L^{\rho q}(\Omega)}\leq C(1+t^{-\frac{\alpha}{2}-\frac{\alpha d}{2q}}),~ ~ q>d.\label{eq:Eb}
  \end{align}
\end{lemma}
\begin{proof}
  The estimates \eqref{eq:Ee}, \eqref{eq:Ef} and \eqref{eq:Eg} were precisely proved
  in \cite[Lemma 3.3]{Costa23} by using the estimates of the heat semigroup under Neumann boundary conditions in Lemma \ref{lem:Heatgroup}.
  We next consider to derive the others. Specifically, \eqref{eq:HG2} in Lemma \ref{lem:Heatgroup} and \eqref{eq:Mainardi} imply \eqref{eq:Ea} as follows
  \begin{align*}
    \left\|E_{\alpha}(t^{\alpha}\Delta)\omega\right\|_{L^{q}(\Omega)}
    &\leq\int_{0}^{\infty}M_{\alpha}(s)\left\|e^{st^{\alpha}\Delta}\omega\right\|_{L^{q}(\Omega)}ds \\
    &\leq\int_{0}^{\infty}M_{\alpha}(s)\left(1+C_2\right)ds\|\omega\|_{L^{q}(\Omega)}
    \leq C\|\omega\|_{L^{q}(\Omega)}.
  \end{align*}

  Regarding \eqref{eq:Ed}, applying \eqref{eq:HG2} in Lemma \ref{lem:Heatgroup}, we have
  \begin{align*}
     \left\|E_{\alpha,\alpha}(t^{\alpha}(\Delta-\gamma))\omega\right\|_{L^{\rho q}(\Omega)}
    &\leq\int_{0}^{\infty}\alpha s M_{\alpha}(s)e^{-\gamma st^{\alpha}}\big\|
         e^{st^{\alpha}\Delta}\omega\big\|_{L^{\rho q}(\Omega)}ds  \\
    &\leq C_2\int_{0}^{\infty}\alpha s M_{\alpha}(s)\big(1+(st^{\alpha})^{-\frac{d}{2 q}}\big)
         \|\omega\|_{L^{\frac{\rho q}{\rho+1}}(\Omega)}ds         \\
    &\leq C_2\alpha\Big(\frac{1}{\Gamma(1+\alpha)}+t^{-\frac{\alpha d}{2q}}
        \frac{\Gamma(2-\frac{d}{2 q})}{\Gamma(1+\alpha(1-\frac{d}{2 q}))}\Big)\|\omega\|_{L^{\frac{\rho q}{\rho+1}}(\Omega)}   \\
    &\leq C(1+t^{-\frac{\alpha d}{2q}})\|\omega\|_{L^{\frac{\rho q}{\rho+1}}(\Omega)}.
  \end{align*}
  By the similar approach, \eqref{eq:Ec} and \eqref{eq:Eb} can be derived by \eqref{eq:HG3} and \eqref{eq:HG4} in Lemma \ref{lem:Heatgroup}, respectively.
\end{proof}

\subsection{Main results}
\label{sec:Well-Regu}
In this subsection, we present the main results on the local existence, uniqueness, and blow-up of the mild solution to the TF-KSNS system \eqref{eq:problem}-\eqref{eq:InitialC} with the chemotactic sensitivity being $\chi(c)=1$. To begin with, the following assumptions are made similar to those for \eqref{eq:KSNS} in \cite{Wang17,Winkler20}.
\begin{assumption}\label{ass:pi}
  Let $q>2d$, $\rho\ge2$ and $\beta=\frac{\alpha d}{2\rho q}$ throughout the paper, and the following two conditions hold.
  \begin{enumerate}[label=(\roman*)]
    \item The time-independent gravitational potential parameter function $\Phi$ satisfies
          $
            \Phi\in W^{1,\infty}(\Omega);
          $
    \item The initial data $(n_0,c_0,\mathbf{u}_0)$ fulfills that: $n_0\in L^{\infty}(\Omega)$ is nonnegative with $n_0\not\equiv0$, $c_0\in W^{1,\infty}(\Omega)$ is nonnegative and $\mathbf{u}_0\in D(\mathcal{A})$.
  \end{enumerate}
\end{assumption}

In Assumption \ref{ass:pi} (ii), $D(\mathcal{A}):=W^{2,q}(\Omega;\mathbb{R}^{d})\cap W_0^{1,q}(\Omega;\mathbb{R}^{d})\cap L^{q}_{\sigma}(\Omega;\mathbb{R}^{d})$ denotes the domain of the Stokes operator $\mathcal{A}:=-\mathcal{P}\Delta$ in the space $L^{q}_{\sigma}(\Omega;\mathbb{R}^{d})$ with $L^{q}_{\sigma}(\Omega;\mathbb{R}^{d}):=\{\phi\in L^{q}(\Omega;\mathbb{R}^d): \nabla\cdot\phi=0\}$ being the subspace of all solenoidal vector fields in $L^{q}(\Omega;\mathbb{R}^d)$, and $\mathcal{P}$ is the Helmholtz projection operator \cite{Mitrea08,Sohr01} from $L^{q}(\Omega;\mathbb{R}^d)$ into $L^{q}_{\sigma}(\Omega;\mathbb{R}^{d})$, which is introduced to remove the pressure term in \eqref{eq:problem1c} and then obtain the solution representations \eqref{eq:DuhamelP}. It holds that $\mathcal{P}[L^{q}(\Omega;\mathbb{R}^{d})]=L_{\sigma}^{q}(\Omega;\mathbb{R}^{d})$, and $\mathcal{P}$ is bounded in $L^p(\Omega)$ for smooth boundary $\partial\Omega$ and any $p\in(1,+\infty)$ \cite[Theorem 1]{FujiwaraM77}.

The local existence and uniqueness of the mild solution to \eqref{eq:problem}-\eqref{eq:InitialC} under the smallness conditions \eqref{eq:intial-c} and \eqref{eq:Tconditions4} are provided in the following theorem.

\begin{theorem}[Well-posedness and $L^q$-regularity]
\label{thm:RegularThm}
  Let $\Omega\subset\mathbb{R}^{d}$ $(d\geq2)$ be a bounded convex domain with a smooth boundary. If the conditions in Assumption \ref{ass:pi} hold, then there exists $T>0$ such that the initial-boundary value problem \eqref{eq:problem}-\eqref{eq:InitialC} admits a unique mild solution $(n,c,\mathbf{u}): [0,T]\rightarrow L^{q}(\Omega;\mathbb{R}^{2+d})$ satisfying
  \begin{equation*}
    n,c\in C([0,T],L^{q}(\Omega)),~\nabla c,\mathbf{u}\in C\big([0,T],L^{q}(\Omega;\mathbb{R}^d)\big),
  \end{equation*}
  and the following asymptotic property holds
  \begin{equation}\label{eq:asymptotic-behav}
    t^{\beta}\|n(t)\|_{L^{q}(\Omega)}\rightarrow0,\quad
    t^{\beta}\|c(t)\|_{L^{q}(\Omega)}\rightarrow0,\quad
    t^{\beta}\|\mathbf{u}(t)\|_{L^{q}(\Omega)}\rightarrow0,\quad {\rm as}~ t\rightarrow0^{+}.
  \end{equation}
  Moreover, the solution $(n,c,\mathbf{u})$ is continuously dependent on the initial data $(n_0,c_0,\mathbf{u}_0)$.
\end{theorem}

With Theorem \ref{thm:RegularThm}, the result of blow-up of the mild solution to \eqref{eq:problem}-\eqref{eq:InitialC} can be further obtained, which is stated as follows.
\begin{theorem}[Blow-up]
\label{thm:Blow-up}
  Under the same conditions in Theorem \ref{thm:RegularThm}, the solution $(n,c,\mathbf{u})$ can be uniquely continued up to maximal time $T_{max}>T$ with
  $
    T_{max}:=\sup\big\{T>0:
      \eqref{eq:problem}-\eqref{eq:InitialC}~{has~ a~ unique~ solution}~(n,c,\mathbf{u})\in C([0,T];L^{\rho q}(\Omega;\mathbb{R}^{2+d}))\big\}.
  $
  Moreover, if $T_{max}<+\infty$, then
  \begin{equation*}
    \underset{t\rightarrow T^{-}_{max}}{\lim\sup}~\|n(t)\|_{L^{\rho q}(\Omega)}=+\infty,~
    \underset{t\rightarrow T^{-}_{max}}{\lim\sup}~\|c(t)\|_{L^{\rho q}(\Omega)}=+\infty,~
    \underset{t\rightarrow T^{-}_{max}}{\lim\sup}~\|\mathbf{u}(t)\|_{L^{\rho q}(\Omega)}=+\infty.
  \end{equation*}
\end{theorem}

The proofs of Theorems \ref{thm:RegularThm} and \ref{thm:Blow-up} are rigorously established in Sections \ref{sec:estmatMS} and \ref{sec:blow-up}, respectively.

\section{Analysis of local well-posedness and regularity}
\label{sec:estmatMS}
\subsection{The mild solution of the TF-KSNS system}
\label{sec:Duhamei}
As $\mathcal{L}_{t\rightarrow z}\{E_{\alpha}(-t^{\beta}A)\}=z^{\beta}/(z^{\beta}+A)$ and $\mathcal{L}_{t\rightarrow z}\{t^{\beta-1}E_{\alpha,\beta}(- t^{\beta}A)\}=z^{\alpha-\beta}/(z^{\alpha}+ A)$ \cite{Gorenflo20}, then taking the Laplace and inverse Laplace transform of the equations in system \eqref{eq:problem} with the projection of \eqref{eq:problem1c} by the Helmholtz projection operator $\mathcal{P}$, it formally derives the following Duhamel-type integral equations
\begin{equation}\label{eq:DuhamelP}
  \left\{\small
  \begin{aligned}
  n(t)&=E_{\alpha}(t^{\alpha}\Delta)n_{0}-\int_{0}^{t}(t-s)^{\alpha-1}
          E_{\alpha,\alpha}\big((t-s)^{\alpha}\Delta\big)
         \big(\mathbf{u}\cdot\nabla n+\nabla\cdot(n\nabla c)\big)(s)ds,\\
  c(t)&=E_{\alpha}(t^{\alpha}(\Delta-\gamma))c_{0}-\int_{0}^{t}(t-s)^{\alpha-1}
          E_{\alpha,\alpha}\big((t-s)^{\alpha}
         (\Delta-\gamma)\big)\big(\mathbf{u}\cdot\nabla c-n\big)(s)ds,\\
  \mathbf{u}(t)&=e^{t\Delta}\mathbf{u}_0-\int_{0}^{t}e^{\Delta (t-s)}\mathcal{P}
         \big(\nabla\cdot(\mathbf{u}\otimes \mathbf{u})(s)-n(s)\nabla\Phi\big)ds,
  \end{aligned}\right.
\end{equation}
where $\mathbf{u}\otimes \mathbf{u}=(u_ju_k)_{j,k=1}^{d}$ with the notation ``$\otimes$" being the usual tensor product, and $\nabla\cdot(\mathbf{u}\otimes \mathbf{u})=(\mathbf{u}\cdot\nabla)\mathbf{u}$ due to the incompressible condition $\nabla\cdot\mathbf{u}=0$.
With \eqref{eq:DuhamelP}, the definition of the mild solution to the TF-KSNS system \eqref{eq:problem}-\eqref{eq:InitialC} reads in the following.
\begin{definition}[$L^{q}(\Omega)$-mild solution]\label{def:mild}
  A triple of continuous functions $(n,c,\mathbf{u}):[0,T]\rightarrow L^{q}(\Omega;\mathbb{R}^{2+d})$ satisfying the Duhamel-type integral equations presented in \eqref{eq:DuhamelP}, is called an $L^{q}(\Omega)$-mild solution to the TF-KSNS system \eqref{eq:problem}-\eqref{eq:InitialC}.
\end{definition}

\subsection{The proof of Theorem \ref{thm:RegularThm}}
\label{sec:estmat}
In light of the aforementioned preparations, in this section, we will employ the Banach fixed point theorem \cite[Theorem 5.7]{Brezis10} to demonstrate the local well-posedness of the mild solution in $L^{q}$ space to the TF-KSNS system \eqref{eq:problem} with the initial-boundary conditions \eqref{eq:BoundaryC}-\eqref{eq:InitialC} satisfying \eqref{eq:intial-c}-\eqref{eq:Tconditions4}, as well as Assumption \ref{ass:pi}. To this end, we introduce a Banach space given by
\begin{equation*}
  \mathbb{S}_{T}:=\big\{(n,c,\mathbf{u})\mid n\in\mathcal{X}, c\in \mathcal{X}, \nabla c\in\mathcal{X}^d, \mathbf{u}\in\mathcal{X}^d\big\}
\end{equation*}
as the solution space, endowed with the norm
\begin{align*}
  \|(n,c,\mathbf{u})\|_{\mathbb{S}_{T}}
  :=&
  \sup_{t\in(0,T)}t^{\beta}\|n(t)\|_{L^{\rho q}(\Omega)}+
  \sup_{t\in(0,T)}t^{\beta}\|c(t)\|_{L^{\rho q}(\Omega)}
\\&+
  \sup_{t\in(0,T)}t^{\beta}\|\nabla c(t)\|_{L^{\rho q}(\Omega)}+
  \sup_{t\in(0,T)}t^{\beta}\|\mathbf{u}(t)\|_{L^{\rho q}(\Omega)},
\end{align*}
where $\mathcal{X}$ is the Banach space
\begin{equation*}
  \mathcal{X}:=\Big\{v\in C((0,T];L^{\rho q}(\Omega))\mid \|v\|_{\mathcal{X}}:=\sup_{t\in(0,T)}t^{\beta}\|v(t)\|_{L^{\rho q}(\Omega)}<+\infty\Big\},
\end{equation*}
with $T\in(0,1)$ being small enough satisfying the conditions in \eqref{eq:intial-c}-\eqref{eq:Tconditions4}.

We denote $\mathbb{X}$ a nonempty complete metric subspace of $\mathbb{S}_{T}$ indicated by
\begin{displaymath}
  \mathbb{X}:=\big\{(n,c,\mathbf{u})\in\mathbb{S}_{T}\mid\|(n,c,\mathbf{u})\|_{\mathbb{S}_{T}}\leq 3R\big\},
\end{displaymath}
then, given $t\in(0,T)$ with sufficiently small $T$, it provides a map $\mathcal{M}=(\mathcal{M}_1, \mathcal{M}_2, \mathcal{M}_3)$ on $\mathbb{X}$ as follows
\begin{equation}\label{eq:map}
  \left\{\small
  \begin{aligned}
  \mathcal{M}_1(n,c,\mathbf{u})&:=E_{\alpha}(t^{\alpha}\Delta)n_{0}
         -\int_{0}^{t}(t-s)^{\alpha-1}E_{\alpha,\alpha}\big((t-s)^{\alpha}\Delta\big)
         \big(\mathbf{u}\cdot\nabla n+\nabla\cdot(n\nabla c)\big)(s)ds,  \\
  \mathcal{M}_2(n,c,\mathbf{u})&:=E_{\alpha}(t^{\alpha}(\Delta-\gamma))c_{0}
         -\int_{0}^{t}(t-s)^{\alpha-1}E_{\alpha,\alpha}\big((t-s)^{\alpha}
         (\Delta-\gamma)\big)\big(\mathbf{u}\cdot\nabla c-n\big)(s)ds,   \\
  \mathcal{M}_3(n,c,\mathbf{u})&:=e^{t\Delta}\mathbf{u}_0-\int_{0}^{t}e^{(t-s)\Delta}
  \mathcal{P}\big(\nabla\cdot(\mathbf{u}\otimes \mathbf{u})(s)-n(s)\nabla\Phi\big)ds.
  \end{aligned}\right.
\end{equation}

Additionally, we take $T>0$ being small enough to satisfy the following inequalities.
\begin{align}
  &T^{\beta}\|\nabla c_0\|_{L^{\rho q}(\Omega)}\le\frac{R}{8C},\label{eq:intial-c}\\
  &T^{\frac{\alpha}{2}}B(1-\beta,\frac{\alpha}{2})+CT^{\alpha}B(1-\beta,\alpha)+CT
      \le\frac{1}{8},  \label{eq:Tconditions0}\\
  &T^{\frac{\alpha}{2}-\frac{\alpha d}{2q}-\beta}
      \le\frac{1}{8CB(1-2\beta,\frac{\alpha}{2}-\frac{\alpha d}{2q})R}, \label{eq:Tconditions01}\\
  &T^{\frac{1}{2}-\frac{d}{2\rho q}-\beta}
      \le\frac{1}{8CB(1-2\beta,\frac{1}{2}-\frac{d}{2\rho q})R},\label{eq:Tconditions02}\\
  &\begin{aligned}
   \sup_{t\in(0,T)}t^{\beta}\left\|E_{\beta}(t^{\alpha}\Delta)n_0\right\|_{L^{\rho q}(\Omega)}
     +\sup_{t\in(0,T)}t^{\beta}\left\|E_{\beta}(t^{\alpha}(\Delta-\gamma))c_0\right\|_{L^{\rho q}(\Omega)}\\
     +\sup_{t\in(0,T)}t^{\beta}\|e^{t\Delta}\mathbf{u}_0\|_{L^{\rho q}(\Omega)}
     \le\frac{R}{8}, \label{eq:Tconditions4}
   \end{aligned}
\end{align}
where \eqref{eq:Tconditions4} is reasonable from Lemmas \ref{lem:MLbounds}, \ref{lem:Dirichletgroup} and Assumption \ref{ass:pi}, and $B(\cdot,\cdot)$ is the beta function defined as $B(a,b):=\int_{0}^{1}s^{a-1}(1-s)^{b-1}ds$.
Based on the aforementioned arrangements, we first derive a preliminary Lemma as follows.

\begin{lemma}[Priori estimates in $L^{\rho q}(\Omega)$]\label{lem:est-integral}
  For $t\in(0,T)$, $n,c\in\mathcal{X}$, and $\mathbf{u},\mathbf{u}_1,\mathbf{u}_2\in \mathcal{X}^d$, we have
  \begin{align}
    &\begin{aligned}\label{eq:I1}
      I_1:&=\int_{0}^{t}(t-s)^{\alpha-1}\big\|E_{\alpha,\alpha}\big((t-s)^{\alpha}\Delta\big)\big(\mathbf{u}\cdot\nabla n\big)(s)\big\|_{L^{\rho q}}ds\\
          &\leq Ct^{\frac{\alpha}{2}-\frac{\alpha d}{2q}-2\beta}
               B(1-2\beta,\frac{\alpha}{2}-\frac{\alpha d}{2q})
               \|n\|_{\mathcal{X}}\|\mathbf{u}\|_{\mathcal{X}},
    \end{aligned}\\
    &\begin{aligned}\label{eq:I2}
      I_2:&=\int_{0}^{t}(t-s)^{\alpha-1}\big\|E_{\alpha,\alpha}\big((t-s)^{\alpha}\Delta\big)
                    \nabla\cdot \big(n\nabla c\big)(s)\big\|_{L^{\rho q}}ds\\
          &\leq Ct^{\frac{\alpha}{2}-\frac{\alpha d}{2q}-2\beta}
               B(1-2\beta,\frac{\alpha}{2}-\frac{\alpha d}{2q})
               \|n\|_{\mathcal{X}}\|\nabla c\|_{\mathcal{X}},
    \end{aligned}\\
    &\begin{aligned}\label{eq:I3}
      I_3:&=\int_{0}^{t}(t-s)^{\alpha-1}\big\|E_{\alpha,\alpha}\big((t-s)^{\alpha}(\Delta-\gamma)\big)\big(\mathbf{u}\cdot\nabla c\big)(s)\big\|_{L^{\rho q}}ds\\
          &\leq Ct^{\alpha-\frac{\alpha d}{2q}-2\beta}B(1-2\beta,\alpha-\frac{\alpha d }{2 q})\|\mathbf{u}\|_{\mathcal{X}}\|\nabla c\|_{\mathcal{X}},
    \end{aligned}\\
    &\begin{aligned}\label{eq:I4}
      I_4:&=\int_{0}^{t}(t-s)^{\alpha-1}\big\|E_{\alpha,\alpha}\big((t-s)^{\alpha}(\Delta-\gamma)\big)n(s)\big\|_{L^{\rho q}}ds\\
          &\leq Ct^{\alpha-\beta}B(1-\beta,\alpha)\|n\|_{\mathcal{X}},
    \end{aligned}\\
    &\begin{aligned}\label{eq:I5}
      I_5:&=\int_{0}^{t}(t-s)^{\alpha-1}\big\|\nabla
                    E_{\alpha,\alpha}\big((t-s)^{\alpha}(\Delta-\gamma)\big)\big(\mathbf{u}\cdot\nabla c\big)(s)\big\|_{L^{\rho q}}ds\\
          &\leq Ct^{\frac{\alpha}{2}-\frac{\alpha d}{2q}-2\beta}B(1-2\beta,\frac{\alpha}{2}-\frac{\alpha d}{2q})\|\mathbf{u}\|_{\mathcal{X}}\|\nabla c\|_{\mathcal{X}},
    \end{aligned}\\
    &\begin{aligned}\label{eq:I6}
      I_6:&=\int_{0}^{t}(t-s)^{\alpha-1}\big\|\nabla E_{\alpha,\alpha}\big((t-s)^{\alpha}(\Delta-\gamma)\big)n(s)\big\|_{L^{\rho q}}ds\\
          &\leq Ct^{\frac{\alpha}{2}-\beta}B(1-\beta,\frac{\alpha}{2})\|n\|_{\mathcal{X}},
    \end{aligned}\\
    &\begin{aligned}\label{eq:I7}
      I_7:&=\int_{0}^{t}\big\|e^{(t-s)\Delta}\mathcal{P}\big(\nabla\cdot(\mathbf{u}_1\otimes \mathbf{u}_2)(s)\big)\big\|_{L^{\rho q}}ds\\
        &\leq Ct^{\frac{1}{2}-\frac{d}{2\rho q}-2\beta}B(1-2\beta,\frac{1}{2}-\frac{d}{2\rho q})\|\mathbf{u}_1\|_{\mathcal{X}}\|\mathbf{u}_2\|_{\mathcal{X}},
    \end{aligned}\\
    &I_8:=\int_{0}^{t}\big\|e^{(t-s)\Delta}\mathcal{P}\big(n(s)\nabla\Phi\big)\big\|_{L^{\rho q}}ds\leq Ct^{1-\beta}\|n\|_{\mathcal{X}}. \label{eq:I8}
  \end{align}
\end{lemma}
\begin{proof}
  As $\mathbf{u}\cdot\nabla n=\nabla\cdot(\mathbf{u}n)$ obtained by the incompressible condition $\nabla\cdot\mathbf{u}=0$, thanks to \eref{eq:Eb} in Lemma \ref{lem:MLbounds} and $L^{\rho q}(\Omega)\hookrightarrow L^{q}(\Omega)$ as $\rho q>q$, it derives that
  \begin{align*}
    I_1&\leq C\int_{0}^{t}(t-s)^{\frac{\alpha}{2}-\frac{\alpha d}{2q}-1}
           \|(\mathbf{u}n)(s)\|_{L^{\frac{\rho q}{\rho+1}}(\Omega)}ds        \\
       &\leq C\int_{0}^{t}(t-s)^{\frac{\alpha}{2}-\frac{\alpha d}{2q}-1}
           \|n(s)\|_{L^{\rho q}(\Omega)}\|\mathbf{u}(s)\|_{L^{q}(\Omega)}ds  \\
       &\leq C\int_{0}^{t}(t-s)^{\frac{\alpha}{2}-\frac{\alpha d}{2q}-1}s^{-2\beta}ds
           \|n\|_{\mathcal{X}}\|\mathbf{u}\|_{\mathcal{X}},
  \end{align*}
  which implies the result \eqref{eq:I1}, and the estimate \eqref{eq:I2} can be similarly derived.

  By \eqref{eq:Ed} in Lemma \ref{lem:MLbounds} and \eqref{eq:EDelta2}, we deduce that
  \begin{align*}
    I_3&\leq C\int_{0}^{t}(t-s)^{\alpha-\frac{\alpha d}{2q}-1}
           \|(\mathbf{u}\cdot\nabla c)(s)\|_{L^{\frac{\rho q}{\rho+1}}(\Omega)}ds \\
       &\leq C\int_{0}^{t}(t-s)^{\alpha-\frac{\alpha d}{2q}-1}s^{-2\beta}ds
           \|\mathbf{u}\|_{\mathcal{X}}\|\nabla c\|_{\mathcal{X}} \\
       &\leq Ct^{\alpha-\frac{\alpha d}{2q}-2\beta}B(1-2\beta,\alpha-\frac{\alpha d}{2q})
           \|\mathbf{u}\|_{\mathcal{X}}\|\nabla c\|_{\mathcal{X}},
  \end{align*}
  which reduces to \eqref{eq:I3},
  and \eqref{eq:Ee} in Lemma \ref{lem:MLbounds} directly follows that
  \begin{align*}
    I_4\leq C\int_{0}^{t}(t-s)^{\alpha-1}s^{-\beta}ds\|n\|_{\mathcal{X}}
       \leq Ct^{\alpha-\beta}B(1-\beta,\alpha)\|n\|_{\mathcal{X}}.
  \end{align*}

  Similarly, we deduce \eqref{eq:I5} by \eqref{eq:Ec} in Lemma \ref{lem:MLbounds} that
  \begin{align*}
    I_5&\leq C\int_{0}^{t}(t-s)^{\frac{\alpha}{2}-\frac{\alpha d}{2q}-1}
         \|(\mathbf{u}\cdot\nabla c)(s)\|_{L^{\frac{\rho q}{\rho+1}}(\Omega)}ds \\
       &\leq C\int_{0}^{t}(t-s)^{\frac{\alpha}{2}-\frac{\alpha d}{2q}-1}s^{-2\beta}ds
         \|\mathbf{u}\|_{\mathcal{X}}\|\nabla c\|_{\mathcal{X}} \\
       &\leq Ct^{\frac{\alpha}{2}-\frac{\alpha d}{2q}-2\beta}
         B(1-2\beta,\frac{\alpha}{2}(1-\frac{d }{q}))\|\mathbf{u}\|_{\mathcal{X}}\|\nabla c\|_{\mathcal{X}},
  \end{align*}
   and by \eqref{eq:Ef} in Lemma \ref{lem:MLbounds}, it follows
   \begin{align*}
    I_6\leq C\int_{0}^{t}(t-s)^{\frac{\alpha}{2}-1}s^{-\beta}ds\|n\|_{\mathcal{X}}
       \leq Ct^{\frac{\alpha}{2}-\beta}B(1-\beta,\frac{\alpha}{2})\|n\|_{\mathcal{X}}.
  \end{align*}

  By duality, it holds that $\|e^{(t-s)\Delta}\mathcal{P}(u\cdot\nabla u)\|_{L^{\rho q}(\Omega)}=\sup_{\|v\|_{(L^{\rho q}(\Omega))'}=1}|(e^{(t-s)\Delta}\mathcal{P}(u\cdot\nabla u),v)|\le\sup_{\|v\|_{(L^{\rho q}(\Omega))'}=1}|(u\cdot\nabla u,e^{(t-s)\Delta}v)|=\|e^{(t-s)\Delta}\nabla\cdot(u\otimes u)\|_{L^{\rho q}(\Omega)}$ (see e.g., \cite{Borchers88}).
  With this, we rely on the result \eqref{eq:DivsemigroupLap} in Lemma \ref{lem:Dirichletgroup} and $\|\mathbf{u}_1\otimes \mathbf{u}_2\|_{L^{\frac{\rho q}{2}}(\Omega)}\leq\|\mathbf{u}_1\|_{L^{\rho q}(\Omega)}\|\mathbf{u}_2\|_{L^{\rho q}(\Omega)}$ to obtain
  \begin{align*}
    I_7&\leq C\int_{0}^{t}(t-s)^{-\frac{1}{2}-\frac{d}{2\rho q}}
         \|(\mathbf{u}_1\otimes \mathbf{u}_2)(s)\|_{L^{\frac{\rho q}{2}}(\Omega)}ds \\
       &\leq C\int_{0}^{t}(t-s)^{-\frac{1}{2}-\frac{d}{2\rho q}}s^{-2\beta}ds
         \|\mathbf{u}_1\|_{\mathcal{X}}\|\mathbf{u}_2\|_{\mathcal{X}} \\
       &\leq Ct^{\frac{1}{2}-\frac{d}{2\rho q}-2\beta}B(1-2\beta,\frac{1}{2}-\frac{d}{2\rho q})\|\mathbf{u}\|^2_{\mathcal{X}},
  \end{align*}
  which obtains \eqref{eq:I7},
  and it finally infers \eqref{eq:I8} directly from \eqref{eq:semigroupLap} in Lemma \ref{lem:Dirichletgroup} that
  \begin{align*}
   I_8\leq C\int_{0}^{t}\big\|(n(s)\nabla\Phi)\big\|_{L^{\rho q}(\Omega)}ds
       \leq C\int_{0}^{t}s^{-\beta}ds\|n\|_{\mathcal{X}}\|\nabla\Phi\|_{L^{\infty}(\Omega)}
       \leq Ct^{1-\beta}\|n\|_{\mathcal{X}}.
  \end{align*}
  The proof is completed.
\end{proof}

We next confirm that the map $\mathcal{M}=(\mathcal{M}_1, \mathcal{M}_2, \mathcal{M}_3)$ in \eqref{eq:map} is well-defined in $\mathbb{X}$ for $t>0$ and maps $\mathbb{X}$ to itself, which can be obtained by the following two lemmas.
\begin{lemma}\label{lem:MiX}
  Let the conditions in Assumption \ref{ass:pi} be satisfied and $T>0$ be small enough satisfying \eqref{eq:intial-c}-\eqref{eq:Tconditions4}. If $(n,c,\mathbf{u})\in\mathbb{X}$, then $\|\mathcal{M}(n,c,\mathbf{u})\|_{\mathbb{S}_T}\leq 3R$.
\end{lemma}
\begin{proof}
  The incompressible condition ($\nabla\cdot\mathbf{u}=0$) implies that $\nabla\cdot(\mathbf{u}n)=\mathbf{u}\cdot\nabla n$. Then, together with the formula of $\mathcal{M}_1$ in \eqref{eq:map}, \eqref{eq:I1} and \eqref{eq:I2} in Lemma \ref{lem:est-integral}, we demonstrate that
  \begin{align*}
   t^{\beta}\|\mathcal{M}_1(n,c,\mathbf{u})\|_{L^{\rho q}(\Omega)}
    &\leq t^{\beta}\|E_{\alpha}(t^{\alpha}\Delta)n_{0}\|_{L^{\rho q}(\Omega)}
        +Ct^{\frac{\alpha}{2}-\frac{\alpha d}{2q}-\beta}
        B(1-2\beta,\frac{\alpha}{2}-\frac{\alpha d}{2q})
        \|n\|_{\mathcal{X}}\|\mathbf{u}\|_{\mathcal{X}}   \\
    &~~~~+Ct^{\frac{\alpha}{2}-\frac{\alpha d}{2q}-\beta}
         B(1-2\beta,\frac{\alpha}{2}-\frac{\alpha d}{2q})
         \|n\|_{\mathcal{X}}\|\nabla c\|_{\mathcal{X}}.
  \end{align*}
  Hence, by the conditions \eqref{eq:Tconditions01} and \eqref{eq:Tconditions4}, we have
  \begin{equation}\label{eq:M1-b}
    \sup_{t\in(0,T)}t^{\beta}\|\mathcal{M}_1(n,c,\mathbf{u})\|_{L^{\rho q}(\Omega)}
    \leq~ \frac{R}{8}+2Ct^{\frac{\alpha}{2}-\frac{\alpha d}{2q}-\beta}
          B(1-2\beta,\frac{\alpha}{2}-\frac{\alpha d}{2q})R^2\leq\frac{3}{8}R.
  \end{equation}

  From \eqref{eq:I3} and \eqref{eq:I4} in Lemma \ref{lem:est-integral}, the map $\mathcal{M}_2(n,c,\mathbf{u})$ in (\ref{eq:map}) satisfies
  \begin{align}\label{eq:M2-b0}
  t^{\beta}\|\mathcal{M}_2(n,c,\mathbf{u})\|_{L^{\rho q}(\Omega)}
       &\leq~t^{\beta}\big\|E_{\alpha}\big(t^{\alpha}(\Delta-\gamma)\big)c_{0}\big\|_{L^{\rho q}(\Omega)}
        +Ct^{\alpha}B(1-\beta,\alpha)\|n\|_{\mathcal{X}}         \notag \\
       &~~~~+Ct^{\alpha-\frac{\alpha d}{2q}-\beta}
        B(1-2\beta,\alpha-\frac{\alpha d }{2 q})\|\mathbf{u}\|_{\mathcal{X}}\|\nabla c\|_{\mathcal{X}}.
  \end{align}
  Since $B(1-2\beta,\frac{\alpha}{2}-\frac{\alpha d}{2q})>B(1-2\beta,\alpha-\frac{\alpha d}{2q})$ for $\alpha\in(0,1)$, then we obtain from \eqref{eq:Tconditions0}, \eqref{eq:Tconditions01}, \eqref{eq:Tconditions4}, and \eqref{eq:M2-b0} that
  \begin{align}\label{eq:M2-b}
    &\sup_{t\in(0,T)}t^{\beta}\|\mathcal{M}_2(n,c,\mathbf{u})\|_{L^{\rho q}(\Omega)} \notag \\
    &\leq \frac{R}{8}+CT^{\alpha}B(1-\beta,\alpha)R+CT^{\alpha-\frac{\alpha d}{2q}-\beta}
                  B(1-2\beta,\alpha-\frac{\alpha d }{2 q})R^2\leq\frac{3}{8}R.
  \end{align}

  According to \eqref{eq:HG1} in Lemma \ref{lem:Heatgroup} and \eqref{eq:Ed} in Lemma \ref{lem:MLbounds}, we can derive $\|\nabla E_{\alpha}(t^{\alpha}(\triangle-\gamma))c_0\|_{L^{\rho q}(\Omega)}\leq C_1\|\nabla c_0\|_{L^{\rho q}(\Omega)}$. With this, by the estimates \eqref{eq:I5} and \eqref{eq:I6} in Lemma \ref{lem:est-integral}, it obtains that
  \begin{align}\label{eq:M2-c}
   t^{\beta}\|\nabla\mathcal{M}_2(n,c,\mathbf{u})\|_{L^{\rho q}(\Omega)}
      &\leq Ct^{\beta}\big\|\nabla c_{0}\big\|_{L^{\rho q}(\Omega)}
           +Ct^{\frac{\alpha}{2}}B(1-\beta,\frac{\alpha}{2})\|n\|_{\mathcal{X}} \notag\\
      &~~~~+Ct^{\frac{\alpha}{2}-\frac{\alpha d }{2 q}-\beta}B(1-2\beta,\frac{\alpha}{2}-\frac{\alpha d}{2q})\|\mathbf{u}\|_{\mathcal{X}}\|\nabla c\|_{\mathcal{X}}.
  \end{align}
  Then, the conditions \eqref{eq:intial-c}, \eqref{eq:Tconditions0} and \eqref{eq:Tconditions01} show that
  \begin{align}\label{eq:M2-d}
     &\sup_{t\in(0,T)}t^{\beta}\|\nabla\mathcal{M}_2(n,c,\mathbf{u})\|_{L^{\rho q}(\Omega)} \notag\\
     &\leq\frac{R}{8}+CT^{\frac{\alpha}{2}}B(1-\beta,\frac{\alpha}{2})R
               +CT^{\frac{\alpha}{2}-\frac{\alpha d }{2 q}-\beta}
                B(1-2\beta,\frac{\alpha}{2}-\frac{\alpha d}{2q})R^2\leq\frac{3}{8}R.
  \end{align}

  Similarly, by the conditions \eqref{eq:Tconditions0}, \eqref{eq:Tconditions02}, \eqref{eq:Tconditions4}, and the estimates \eqref{eq:I7} and \eqref{eq:I8} in Lemma \ref{lem:est-integral}, we infer that
  \begin{align}\label{eq:M3-b}
    &\sup_{t\in(0,T)}t^{\beta}\|\mathcal{M}_3(n,c,\mathbf{u})\|_{L^{\rho q}(\Omega)} \notag \\
    &\leq \frac{R}{8}+CT^{\frac{1}{2}-\frac{d}{2\rho q}-\beta}
      B(1-2\beta,\frac{1}{2}-\frac{d}{2\rho q})R^2 + CTR\leq\frac{3}{8}R.
  \end{align}

  Thus, the result is obtained from the estimates \eqref{eq:M1-b}, \eqref{eq:M2-b}, \eqref{eq:M2-d} and \eqref{eq:M3-b}.
\end{proof}

\begin{lemma}[Continuity]\label{lem:cMt}
  Under the conditions in Lemma \ref{lem:MiX}, the map $\mathcal{M}=(\mathcal{M}_1, \mathcal{M}_2, \mathcal{M}_3)$ in \eqref{eq:map} is continuous with respect to variable $t\in(0,T]$ in $\mathbb{X}$.
\end{lemma}
\begin{proof}
  It suffices to derive $\mathcal{M}_i(n,c,\mathbf{u})(t)\in C((0,T];L^{\rho q}(\Omega))$ if $(n,c,\mathbf{u})\in\mathbb{X}$, $i=1,2,3$.
  We first prove the continuity of $\mathcal{M}_1$ in \eref{eq:map}. Let $0<t_1<t_2\le T$, it yields
  \begin{equation}\label{eq:M1}
    \mathcal{M}_1(n,c,\mathbf{u})(t_2)-\mathcal{M}_1(n,c,\mathbf{u})(t_1)=\big(E_{\alpha}(t_2^{\alpha}\Delta)-E_{\alpha}(t_1^{\alpha}\Delta)\big)n_{0}+\sum_{k=1}^{6}I^{\mathcal{M}_1}_{k},
  \end{equation}
  where
  \begin{align*}
    I^{\mathcal{M}_1}_{1}:=&\int_{t_1}^{t_2}(t_2-s)^{\alpha-1}E_{\alpha,\alpha}\big((t_2-s)^{\alpha}\Delta\big)(\mathbf{u}\cdot\nabla n)(s)ds, \\
    I^{\mathcal{M}_1}_{2}:=&\int_{t_1}^{t_2}(t_2-s)^{\alpha-1}E_{\alpha,\alpha}\big((t_2-s)^{\alpha}\Delta\big)\nabla\cdot(n\nabla c)(s)ds, \\
    I^{\mathcal{M}_1}_{3}:=&\int_{0}^{t_1}\big((t_2-s)^{\alpha-1}-(t_1-s)^{\alpha-1}\big)E_{\alpha,\alpha}\big((t_2-s)^{\alpha}\Delta\big)(\mathbf{u}\cdot\nabla n)(s)ds, \\
    I^{\mathcal{M}_1}_{4}:=&\int_{0}^{t_1}(t_1-s)^{\alpha-1}\big[E_{\alpha,\alpha}\big((t_2-s)^{\alpha}\Delta\big)-E_{\alpha,\alpha}\big((t_1-s)^{\alpha}\Delta\big)\big](\mathbf{u}\cdot\nabla n)(s)ds,\\
    I^{\mathcal{M}_1}_{5}:=&\int_{0}^{t_1}\big((t_2-s)^{\alpha-1}-(t_1-s)^{\alpha-1}\big)
        E_{\alpha,\alpha}\big((t_2-s)^{\alpha}\Delta\big)\nabla\cdot(n\nabla c)(s)ds, \\
    I^{\mathcal{M}_1}_{6}:=&\int_{0}^{t_1}(t_1-s)^{\alpha-1}\big[E_{\alpha,\alpha}
        \big((t_2-s)^{\alpha}\Delta\big)-E_{\alpha,\alpha}\big((t_1-s)^{\alpha}\Delta\big)\big]
        \nabla\cdot(n\nabla c)(s)ds.
  \end{align*}
  As we know from Lemma \ref{lem:continuous} that $E_{\alpha}(t^{\alpha}\Delta)$ is strongly continuous, then the term $\big(E_{\alpha}(t_2^{\alpha}\Delta)-E_{\alpha}(t_1^{\alpha}\Delta)\big)n_{0}$ in \eqref{eq:M1} tends to $0$ in $L^{\rho q}(\Omega)$ as $t_1\rightarrow t_{2}^{-}$.
  Analogous to the estimates \eqref{eq:I1} and \eqref{eq:I2} in Lemma \ref{lem:est-integral}, we have
  \begin{align}
    \big\|I^{\mathcal{M}_1}_{1}\big\|_{L^{\rho q}(\Omega)}
     &\leq\int_{t_1}^{t_2}(t_2-s)^{\alpha-1}\big\|E_{\alpha,\alpha}\big((t_2-s)^{\alpha}\Delta\big)
      \nabla\cdot(\mathbf{u}n)(s)\big\|_{L^{\rho q}(\Omega)}ds \notag \\
     &\leq C\int_{t_1}^{t_2}(t_2-s)^{\frac{\alpha}{2}-\frac{\alpha d}{2q}-1}s^{-2\beta}ds
       \|n\|_{\mathcal{X}}\|\mathbf{u}\|_{\mathcal{X}} \notag \\
     &\leq C t_2^{\frac{\alpha}{2}-\frac{\alpha d}{2q}-2\beta}R^2
       \int_{\frac{t_1}{t_2}}^{1}(1-\tau)^{\frac{\alpha}{2}-\frac{\alpha d}{2q}-1}\tau^{-2\beta}d\tau,\label{eq:IF1}
  \end{align}
  and
  \begin{align}
    \big\|I^{\mathcal{M}_1}_{2}\big\|_{L^{\rho q}(\Omega)}
    &\leq C\int_{t_1}^{t_2}(t_2-s)^{\frac{\alpha}{2}-\frac{\alpha d}{2q}-1}s^{-2\beta}ds
          \|n\|_{\mathcal{X}}\|\nabla c\|_{\mathcal{X}} \notag \\
    &\leq C t_2^{\frac{\alpha}{2}-\frac{\alpha d}{2q}-2\beta}R^2
         \int_{\frac{t_1}{t_2}}^{1}(1-\tau)^{\frac{\alpha}{2}-\frac{\alpha d}{2q}-1}\tau^{-2\beta}d\tau.\label{eq:IF2}
  \end{align}
  As $t_1\rightarrow t_{2}^{-}$, it is evident that $\|I^{\mathcal{M}_1}_{1}\|_{L^{\rho q}(\Omega)}$ and $\|I^{\mathcal{M}_1}_{2}\|_{L^{\rho q}(\Omega)}$ converge to zero.

  By \eqref{eq:Eb} in Lemma \ref{lem:MLbounds}, it follows that
  \begin{align*}
    \big\|I^{\mathcal{M}_1}_{3}\big\|_{L^{\rho q}(\Omega)}
    &\leq \int_{0}^{t_1}\big|(t_2-s)^{\alpha-1}-(t_1-s)^{\alpha-1}\big|
       \big\|E_{\alpha,\alpha}
       \big((t_2-s)^{\alpha}\Delta\big)\nabla\cdot(\mathbf{u}n)(s)\big\|_{L^{\rho q}(\Omega)}ds  \\
    &\leq C R^2\int_{0}^{t_1}\big|(t_2-s)^{\alpha-1}-(t_1-s)^{\alpha-1}\big|
      (t_2-s)^{-\frac{\alpha}{2}-\frac{\alpha d}{2q}}s^{-2\beta}ds,
  \end{align*}
  \begin{align*}
    \big\|I^{\mathcal{M}_1}_{5}\big\|_{L^{\rho q}(\Omega)}
      \leq C R^2\int_{0}^{t_1}\big|(t_2-s)^{\alpha-1}-(t_1-s)^{\alpha-1}\big|
      (t_2-s)^{-\frac{\alpha}{2}-\frac{\alpha d}{2q}}s^{-2\beta}ds.
  \end{align*}
  Moreover, utilizing \eqref{eq:HG4} in Lemma \ref{lem:Heatgroup} and \eqref{eq:EDelta}-\eqref{eq:EDelta2}, we obtain that
  \begin{align*}
    &\big\|I^{\mathcal{M}_1}_{4}\big\|_{L^{\rho q}(\Omega)}\\
    &\leq \int_{0}^{t_1}(t_1-s)^{\alpha-1}\big\|\big[E_{\alpha,\alpha}
        \big((t_2-s)^{\alpha}\Delta\big)
       -E_{\alpha,\alpha}\big((t_1-s)^{\alpha}\Delta\big)\big]\nabla\cdot(\mathbf{u}n)(s)\big\|_{L^{\rho q}(\Omega)}ds,\\
   &\leq \int_{0}^{t_1}(t_1-s)^{\alpha-1}\int_{0}^{\infty}\alpha \tau M_{\alpha}(\tau)
       \big\|e^{(t_1-s)^{\alpha}\tau\Delta}\nabla\cdot\big(e^{((t_2-s)^{\alpha}-(t_1-s)^{\alpha})\tau\Delta}-I\big)\mathbf{u}n(s)\big\|_{L^{\rho q}(\Omega)}d\tau ds                     \\
    &\leq C\int_{0}^{t_1}(t_1-s)^{\alpha-1}\int_{0}^{\infty}\alpha \tau M_{\alpha}(\tau)
        \big(1+\tau^{-\frac{1}{2}-\frac{d}{2q}}(t_1-s)^{-\frac{\alpha}{2}-\frac{\alpha d}{2q}}\big)\\
    &{\hskip4.0cm}\times\big\|\big(e^{((t_2-s)^{\alpha}-(t_1-s)^{\alpha})\tau\Delta}-I\big)\mathbf{u}n(s)\big\|_{L^{\frac{\rho q}{\rho+1}}(\Omega)}d\tau ds  \\
    &\leq CR^2\int_{0}^{t_1}
      (t_1-s)^{\alpha-1}s^{-2\beta}\int_{0}^{\infty}\alpha \tau M_{\alpha}(\tau)\big(1+\tau^{-\frac{1}{2}-\frac{d}{2q}}(t_1-s)^{-\frac{\alpha}{2}-\frac{\alpha d}{2q}}\big) \\
    &{\hskip5.0cm}\times\big\|e^{((t_2-s)^{\alpha}-(t_1-s)^{\alpha})\tau\Delta} -I\big\|_{L^{\frac{\rho q}{\rho+1}}(\Omega)\rightarrow L^{\frac{\rho q}{\rho+1}}(\Omega)}d\tau ds,
  \end{align*}
  and
  \begin{align*}
  \big\|I^{\mathcal{M}_1}_{6}\big\|_{L^{\rho q}(\Omega)}
  \leq&~ CR^2\int_{0}^{t_1}(t_1-s)^{\alpha-1}s^{-2\beta}
        \int_{0}^{\infty}\alpha \tau M_{\alpha}(\tau)\big(1+\tau^{-\frac{1}{2}-\frac{d}{2q}}(t_1-s)^{-\frac{\alpha}{2}-\frac{\alpha d}{2q}}\big) \\
     &{\hskip1.5cm}\times\big\|e^{((t_2-s)^{\alpha}-(t_1-s)^{\alpha})\tau\Delta}-I\big\|_{L^{\frac{\rho q}{\rho+1}}(\Omega)\rightarrow L^{\frac{\rho q}{\rho+1}}(\Omega)}d\tau ds.
  \end{align*}
  It is easy to demonstrate that $\|I^{\mathcal{M}_1}_{i}\|_{L^{\rho q}(\Omega)}$ ($i=3,4,5,6$) also converge to zero as $t_1\rightarrow t_{2}^{-}$ by the Lebesgue dominated convergence theorem, \eqref{eq:Mainardi} and the continuity of the heat semigroup $e^{t\Delta}$. Thus, we have $\mathcal{M}_1(n,c,\mathbf{u})(t)\in C((0,T],L^{\rho q}(\Omega))$.

  Analogous to the analysis of $\mathcal{M}_1(n,c,\mathbf{u})(t)$, we can also obtain that $\mathcal{M}_2(n,c,\mathbf{u})(t)$ $\in C((0,T],L^{\rho q}(\Omega))$ and $\nabla\mathcal{M}_2(n,c,\mathbf{u})(t)\in C((0,T],L^{\rho q}(\Omega))$ by using Lemma \ref{lem:continuous} and the similar techniques for \eqref{eq:I3} and \eqref{eq:I4} in Lemma \ref{lem:est-integral}.

  Next, we demonstrate the continuity of $\mathcal{M}_3(n,c,\mathbf{u})(t)$. It follows from \eqref{eq:map} that
  \begin{align}
    &\mathcal{M}_3(n,c,\mathbf{u})(t_2)-\mathcal{M}_3(n,c,\mathbf{u})(t_1)  \notag \\
    &=\int_{t_1}^{t_2}e^{(t_2-s)\Delta}\mathcal{P}
       \big(\nabla\cdot(\mathbf{u}\otimes \mathbf{u})-n\nabla\Phi\big)ds   \notag \\
    &~~~~+\int_{0}^{t_1}\big[e^{(t_2-s)\Delta}-e^{(t_1-s)\Delta}\big]
       \mathcal{P}\big(\nabla\cdot(\mathbf{u}\otimes \mathbf{u})-n\nabla\Phi\big)ds.
  \end{align}
  Then it suffices to estimate the four terms in \eqref{eq:c1M3}-\eqref{eq:c4M3} due to the boundedness of the projection $\mathcal{P}$ in $L^{p}(\Omega)$ for $p\in(1,\infty)$. Using the estimates \eqref{eq:semigroupLap}, \eqref{eq:DivsemigroupLap} in Lemma \ref{lem:Dirichletgroup} and $\|(\mathbf{u}\otimes \mathbf{u})\|_{L^{\frac{\rho q}{2}}(\Omega)}=\|\mathbf{u}\|_{L^{\rho q}(\Omega)}^2$, it follows that
  \begin{align}
    \int_{t_1}^{t_2}\|e^{(t_2-s)\Delta}\nabla\cdot(\mathbf{u}\otimes \mathbf{u})\|_{L^{\rho q}(\Omega)}ds
    &\leq C\int_{t_1}^{t_2}(t_2-s)^{-\frac{1}{2}-\frac{d}{2\rho q}}
      \|\mathbf{u}\otimes \mathbf{u}\|_{L^{\frac{\rho q}{2}}(\Omega)}ds  \notag \\
    &\leq Ct_2^{\frac{1}{2}-\frac{d}{2\rho q}-2\beta}\|\mathbf{u}\|^{2}_{\mathcal{X}}
      \int_{\frac{t_1}{t_2}}^{1}(1-\tau)^{-\frac{1}{2}-\frac{d}{2\rho q}}\tau^{-2\beta}d\tau,\label{eq:c1M3}
  \end{align}
  \begin{align}
    \int_{t_1}^{t_2}\|e^{(t_2-s)\Delta}(n\nabla\Phi)\|_{L^{\rho q}(\Omega)}ds
    &\leq C\int_{\frac{t_1}{t_2}}^{1}\|n\nabla\Phi\|_{L^{\rho q}}d\tau \notag \\
    &\leq Ct_2^{1-\beta}\|n\|_{\mathcal{X}}
     \int_{\frac{t_1}{t_2}}^{1}\tau^{-\beta}d\tau\|\nabla\Phi\|_{L^{\infty}(\Omega)},
  \end{align}
  both of which tend to zero as $t_1\rightarrow t_{2}^{-}$.
  Employing \eqref{eq:DivsemigroupLap} in Lemma \ref{lem:Dirichletgroup}, we obtain
  \begin{align*}
    &\int_{0}^{t_1}\big\|\big(e^{(t_2-s)\Delta}-e^{(t_1-s)\Delta}\big)
        \nabla\cdot(\mathbf{u}\otimes \mathbf{u})\big\|_{L^{\rho q}(\Omega)}ds \\
    &\leq C\int_{0}^{t_1}(t_1-s)^{-\frac{1}{2}-\frac{d}{2\rho q}}
        \left\|\left(e^{(t_2-s)\Delta-(t_1-s)\Delta}-I\right)
        (\mathbf{u}\otimes \mathbf{u})\right\|_{L^{\frac{\rho q}{2}}(\Omega)}ds   \\
   &\leq Ct_1^{\frac{1}{2}-\frac{d}{2\rho q}-2\beta}\|\mathbf{u}\|^{2}_{\mathcal{X}}
        \int_{0}^{1}(1-\tau)^{-\frac{1}{2}-\frac{d}{2\rho q}}\tau^{-2\beta}
        \|e^{(t_2-t_1)\Delta}-I\|_{L^{\rho q}(\Omega)\rightarrow L^{\rho q}(\Omega)}d\tau,
  \end{align*}
  and
  \begin{align}\label{eq:c4M3}
    &\int_{0}^{t_1}\big\|\big(e^{(t_2-s)\Delta}-e^{(t_1-s)\Delta}\big)
       (n\nabla\Phi)\big\|_{L^{\rho q}(\Omega)}ds \cr
    &\leq Ct_1^{1-\beta}\|n\|_{\mathcal{X}} \int_{0}^{1}\tau^{-\beta}
       \|e^{(t_2-t_1)\Delta}-I\|_{L^{\rho q}(\Omega)\rightarrow L^{\rho q}(\Omega)}d\tau.
  \end{align}
  Applying the Lebesgue dominated convergence theorem, the above two terms converge to zero as $t_1\rightarrow t_{2}^{-}$. Then we deduce that $\mathcal{M}_3(n,c,\mathbf{u})(t)\in C((0,T],L^{\rho q}(\Omega))$.
\end{proof}

  Combining the results in Lemmas \ref{lem:MiX} and \ref{lem:cMt}, it concludes that $\mathcal{M}:\mathbb{X}\rightarrow\mathbb{X}$ is a well-defined map for $t>0$.
\begin{lemma}\label{lem:continuity}
  The map $\mathcal{M}$ is well-defined in $\mathbb{X}$ for $t>0$ and maps $\mathbb{X}$ to itself.
\end{lemma}

To obtain the local existence and uniqueness of the solution to the TF-KSNS system \eqref{eq:problem}-\eqref{eq:InitialC}, we further need to show the map $\mathcal{M}:\mathbb{X}\rightarrow\mathbb{X}$ is contractive such that it meets the conditions of the Banach fixed point theorem \cite[Theorem 5.7]{Brezis10}.
\begin{lemma}[Contraction]\label{lem:Contraction}
  The map $\mathcal{M}:\mathbb{X}\rightarrow\mathbb{X}$ is a contraction.
\end{lemma}
\begin{proof}
  Let $(n_1, c_1, \mathbf{u}_1),~(n_2, c_2, \mathbf{u}_2)\in\mathbb{X}$, and their distance in $\mathbb{X}$ be given by
  \begin{displaymath}
  \mathcal{D}_{T}[(n_1, c_1, \mathbf{u}_1),(n_2, c_2, \mathbf{u}_2)]
    :=\|n_1-n_2\|_{\mathcal{X}}
    +\|c_1-c_2\|_{\mathcal{X}}
    +\|\nabla(c_1-c_2)\|_{\mathcal{X}}
    +\|\mathbf{u}_1-\mathbf{u}_2\|_{\mathcal{X}}.
  \end{displaymath}

  It yields from \eqref{eq:map} that
  \begin{align*}
    &\big\|\mathcal{M}_1(n_1, c_1, \mathbf{u}_1)-\mathcal{M}_1(n_2, c_2, \mathbf{u}_2)\big\|_{\mathcal{X}} \\
    &=\sup_{t\in(0,T)}t^{\beta}\int_{0}^{t}(t-s)^{\alpha-1}\|E_{\alpha,\alpha}((t-s)^{\alpha}\Delta)\big((\mathbf{u}_2-\mathbf{u}_1)\cdot\nabla n_2\big)(s)\|_{L^{\rho q}(\Omega)}ds  \\
    &~~~~+\sup_{t\in(0,T)}t^{\beta}\int_{0}^{t}(t-s)^{\alpha-1}\|E_{\alpha,\alpha}((t-s)^{\alpha}\Delta)\big(\mathbf{u}_1\cdot(\nabla n_2-\nabla n_1)\big)(s)\|_{L^{\rho q}(\Omega)}ds \\
    &~~~~+\sup_{t\in(0,T)}t^{\beta}\int_{0}^{t}(t-s)^{\alpha-1}\|E_{\alpha,\alpha}((t-s)^{\alpha}\Delta)\nabla\cdot\big((n_2- n_1)\nabla c_2\big)(s)\|_{L^{\rho q}(\Omega)}ds   \\
    &~~~~+\sup_{t\in(0,T)}t^{\beta}\int_{0}^{t}(t-s)^{\alpha-1}\|E_{\alpha,\alpha}((t-s)^{\alpha}\Delta)\nabla\cdot \big(n_1(\nabla c_2-\nabla c_1)\big)(s)\|_{L^{\rho q}(\Omega)}ds \\
    &=J^{\mathcal{M}_1}_{1}+J^{\mathcal{M}_1}_{2}+J^{\mathcal{M}_1}_{3}+J^{\mathcal{M}_1}_{4}.
  \end{align*}

  To begin with, by \eqref{eq:I1} in Lemma \ref{lem:est-integral} and (\ref{eq:Tconditions01}), $J^{\mathcal{M}_1}_{1}$ and $J^{\mathcal{M}_1}_{2}$ are estimated as
  \begin{align*}
    J^{\mathcal{M}_1}_{1}
    &\leq C\sup_{t\in(0,T)}t^{\beta} \int_{0}^{t}(t-s)^{\frac{\alpha}{2}-\frac{\alpha d}{2q}-1}s^{-2\beta}ds\|n_2\|_{\mathcal{X}}\|\mathbf{u}_1-\mathbf{u}_2\|_{\mathcal{X}}  \\
    &\leq CT^{\frac{\alpha}{2}-\frac{\alpha d}{2q}-\beta}B(1-2\beta,\frac{\alpha}{2}-\frac{\alpha d}{2q})R
    \|\mathbf{u}_1-\mathbf{u}_2\|_{\mathcal{X}}
    \leq \frac{1}{8}\|\mathbf{u}_1-\mathbf{u}_2\|_{\mathcal{X}},
  \end{align*}
  and
  \begin{align*}
    J^{\mathcal{M}_1}_{2}
    \leq CT^{\frac{\alpha}{2}-\frac{\alpha d}{2q}-\beta}
         B(1-2\beta,\frac{\alpha}{2}-\frac{\alpha d}{2q})R\|n_1-n_2\|_{\mathcal{X}}
    \leq \frac{1}{8}\|n_1-n_2\|_{\mathcal{X}}.
  \end{align*}
  As to $J^{\mathcal{M}_1}_{3}$ and $J^{\mathcal{M}_1}_{4}$, it derives from \eqref{eq:I2} in Lemma \ref{lem:est-integral} and \eqref{eq:Tconditions01} that
  \begin{align*}
    J^{\mathcal{M}_1}_{3}
    &\leq C\sup_{t\in(0,T)}t^{\beta} \int_{0}^{t}(t-s)^{\frac{\alpha}{2}-\frac{\alpha d}{2q}-1}
          s^{-2\beta}ds\|\nabla c_2\|_{\mathcal{X}}\|n_1-n_2\|_{\mathcal{X}}\\
    &\leq CT^{\frac{\alpha}{2}-\frac{\alpha d}{2q}-\beta}
          B(1-2\beta,\frac{\alpha}{2}-\frac{\alpha d}{2q})R\|n_1-n_2\|_{\mathcal{X}}
     \leq \frac{1}{8}\|n_1-n_2\|_{\mathcal{X}},\\
    J^{\mathcal{M}_1}_{4}
    &\leq CT^{\frac{\alpha}{2}-\frac{\alpha d}{2q}-\beta}
          B(1-2\beta,\frac{\alpha}{2}-\frac{\alpha d}{2q})R\|\nabla c_1-\nabla c_2\|_{\mathcal{X}}
     \leq \frac{1}{8}\|\nabla c_1-\nabla c_2\|_{\mathcal{X}}.
  \end{align*}
  By combining the above four estimates, we obtain
  \begin{equation}\label{eq:M1-nu}
    \big\|\mathcal{M}_1(n_1, c_1, \mathbf{u}_1)-\mathcal{M}_1(n_2, c_2, \mathbf{u}_2)\big\|_{\mathcal{X}}
    \leq\frac{1}{4}\|n_1-n_2\|_{\mathcal{X}}+\frac{1}{8}\|\nabla c_1-\nabla c_2\|_{\mathcal{X}}+\frac{1}{8}\|\mathbf{u}_1-\mathbf{u}_2\|_{\mathcal{X}}.
  \end{equation}

  Now, it turns to the estimates for $\mathcal{M}_2$ and $\nabla\mathcal{M}_2$. In a similar way, we can conclude from \eqref{eq:I3} and \eqref{eq:I4} in Lemma \ref{lem:est-integral}, the conditions \eqref{eq:Tconditions0} and \eqref{eq:Tconditions01} that
  \begin{align}\label{eq:M2-nu}
  &\big\|\mathcal{M}_2(n_1, c_1, \mathbf{u}_1)-\mathcal{M}_2(n_2, c_2, \mathbf{u}_2)\big\|_{\mathcal{X}}
  \notag\\
  &\leq \sup_{t\in(0,T)}t^{\beta}\int_{0}^{t}(t-s)^{\alpha-1}
      \|E_{\alpha,\alpha}((t-s)^{\alpha}(\Delta-\gamma))\big((\mathbf{u}_1-\mathbf{u}_2)\cdot \nabla c_1\big)(s)\|_{L^{\rho q}(\Omega)}ds   \notag \\
 &~~~~+\sup_{t\in(0,T)}t^{\beta}\int_{0}^{t}(t-s)^{\alpha-1}
      \|E_{\alpha,\alpha}((t-s)^{\alpha}(\Delta-\gamma))\big(\mathbf{u}_2\cdot\nabla(c_2-c_1)\big)(s)\|_{L^{\rho q}(\Omega)}ds  \notag \\
 &~~~~+\sup_{t\in(0,T)}t^{\beta}\int_{0}^{t}(t-s)^{\alpha-1}
      \|E_{\alpha,\alpha}((t-s)^{\alpha}(\Delta-\gamma))(n_2-n_1)(s)\|_{L^{\rho q}(\Omega)}ds \notag \\
 &\leq CT^{\alpha-\frac{\alpha d}{2q}-\beta}B(1-2\beta,\alpha-\frac{\alpha d }{2 q})
      \|\mathbf{u}_1-\mathbf{u}_2\|_{\mathcal{X}}\|\nabla c_1\|_{\mathcal{X}} \notag \\
 &~~~~+CT^{\alpha-\frac{\alpha d}{2q}-\beta}B(1-2\beta,\alpha-\frac{\alpha d }{2 q})
      \|\mathbf{u}_2\|_{\mathcal{X}}\|\nabla(c_1-c_2)\|_{\mathcal{X}} \notag \\
 &~~~~+CT^{\alpha}B(1-\beta,\alpha)\|n_1-n_2\|_{\mathcal{X}} \notag \\
 &\leq \frac{1}{8}\big(\|n_1-n_2\|_{\mathcal{X}}+\|\nabla c_1-\nabla c_2\|_{\mathcal{X}}
      +\|\mathbf{u}_1-\mathbf{u}_2\|_{\mathcal{X}}\big).
  \end{align}
  Similarly, it derives from \eqref{eq:I5} and \eqref{eq:I6}, the conditions \eqref{eq:Tconditions0} and (\ref{eq:Tconditions01}) that
  \begin{align}\label{eq:M2-gnu}
    &\big\|\nabla\mathcal{M}_2(n_1, c_1, \mathbf{u}_1)-\nabla\mathcal{M}_2(n_2, c_2, \mathbf{u}_2)\big\|_{\mathcal{X}} \notag \\
    &\leq \sup_{t\in(0,T)}t^{\beta}\int_{0}^{t}(t-s)^{\alpha-1}
      \|\nabla E_{\alpha,\alpha}((t-s)^{\alpha}(\Delta-\gamma))\big((\mathbf{u}_1-\mathbf{u}_2)\cdot \nabla c_1\big)(s)\|_{L^{\rho q}(\Omega)}ds \notag \\
    &~~~~+\sup_{t\in(0,T)}t^{\beta}\int_{0}^{t}(t-s)^{\alpha-1}
      \|\nabla E_{\alpha,\alpha}((t-s)^{\alpha}(\Delta-\gamma))
      \big(\mathbf{u}_2\cdot\nabla(c_2-c_1)\big)(s)\|_{L^{\rho q}(\Omega)}ds \notag \\
    &~~~~+\sup_{t\in(0,T)}t^{\beta}\int_{0}^{t}(t-s)^{\alpha-1}
      \|\nabla E_{\alpha,\alpha}((t-s)^{\alpha}(\Delta-\gamma))(n_2-n_1)(s)\|_{L^{\rho q}(\Omega)}ds \notag \\
    &\leq  CT^{\frac{\alpha}{2}-\frac{\alpha d}{2q}
         -\beta}B(1-2\beta,\frac{\alpha}{2}-\frac{\alpha d}{2q})\|\nabla c_1\|_{\mathcal{X}}\|\mathbf{u}_1-\mathbf{u}_2\|_{\mathcal{X}} \notag \\
    &~~~~+CT^{\frac{\alpha}{2}-\frac{\alpha d}{2q}-\beta}
         B(1-2\beta,\frac{\alpha}{2}-\frac{\alpha d}{2q})
         \|\mathbf{u}_2\|_{\mathcal{X}}\|\nabla(c_1-c_2)\|_{\mathcal{X}} \notag \\
    &~~~~+CT^{\frac{\alpha}{2}}B(1-\beta,\frac{\alpha}{2})\|n_1-n_2\|_{\mathcal{X}} \notag \\
    &\leq \frac{1}{8}\big(\|n_1-n_2\|_{\mathcal{X}}+\|\nabla c_1-\nabla c_2\|_{\mathcal{X}}
        +\|\mathbf{u}_1-\mathbf{u}_2\|_{\mathcal{X}}\big).
  \end{align}

  Regarding $\mathcal{M}_3(n_1, c_1, \mathbf{u}_1)-\mathcal{M}_3(n_2, c_2, \mathbf{u}_2)$, there exists
  \begin{align}
    &\big\|\mathcal{M}_3(n_1, c_1, \mathbf{u}_1)-\mathcal{M}_3(n_2, c_2, \mathbf{u}_2)\big\|_{\mathcal{X}} \notag \\
    &\leq \sup_{t\in(0,T)}t^{\beta}\int_{0}^{t}\big\|e^{(t-s)\Delta}\mathcal{P}\nabla\cdot
      \big((\mathbf{u}_1-\mathbf{u}_2)\otimes\mathbf{u}_1\big)(s)\big\|_{L^{\rho q}(\Omega)}ds  \notag \\
    &~~~~+\sup_{t\in(0,T)}t^{\beta}\int_{0}^{t}\big\|e^{(t-s)\Delta}\mathcal{P}\nabla\cdot
      \big(\mathbf{u}_2\otimes(\mathbf{u}_1-\mathbf{u}_2)\big)(s)\big\|_{L^{\rho q}(\Omega)}ds  \notag \\
    &~~~~+\sup_{t\in(0,T)}t^{\beta}\int_{0}^{t}
      \big\|e^{(t-s)\Delta}\mathcal{P}\big((n_2-n_1)(s)\nabla\Phi\big)\big\|_{L^{\rho q}(\Omega)}ds \notag \\
    &\leq CT^{\frac{1}{2}-\frac{d}{2\rho q}-\beta}R
     B(1-2\beta, \frac{1}{2}-\frac{d}{2\rho q})\|\mathbf{u}_1-\mathbf{u}_2\|_{\mathcal{X}} \notag \\
    &~~~~+CT^{\frac{1}{2}-\frac{d}{2\rho q}-\beta}R
     B(1-2\beta, \frac{1}{2}-\frac{d}{2\rho q})\|\mathbf{u}_1-\mathbf{u}_2\|_{\mathcal{X}}
       +CT\|n_1-n_2\|_{\mathcal{X}}  \notag \\
    &\leq\frac{1}{8}\|n_1-n_2\|_{\mathcal{X}}+\frac{1}{4}\|\mathbf{u}_1-\mathbf{u}_2\|_{\mathcal{X}}.\label{eq:M3-nu}
  \end{align}

  Combining all of the aforementioned estimates \eqref{eq:M1-nu}, \eqref{eq:M2-nu}, \eqref{eq:M2-gnu}, and (\ref{eq:M3-nu}), it is clear that
  \begin{align*}
    \mathcal{D}_{T}[\mathcal{M}(n_1, v_1, \mathbf{u}_1),\mathcal{M}(n_2, v_2, \mathbf{u}_2)]
    &\leq\frac{5}{8}\|n_1-n_2\|_{\mathcal{X}}+\frac{3}{8}\|\nabla c_1-\nabla c_2\|_{\mathcal{X}}
    +\frac{5}{8}\|\mathbf{u}_1-\mathbf{u}_2\|_{\mathcal{X}}  \\
    &\leq\frac{5}{8}\mathcal{D}_{T}[(n_1, c_1, \mathbf{u}_1),(n_2, c_2, \mathbf{u}_2)],
  \end{align*}
   which asserts that $\mathcal{M}:\mathbb{X}\rightarrow \mathbb{X}$ is a strict contraction.
\end{proof}

With the results of Lemmas \ref{lem:continuity} and \ref{lem:Contraction} established at hand, we can determine the local existence and uniqueness of the solution to the TF-KSNS system \eqref{eq:problem}-\eqref{eq:InitialC}.
\begin{proof}[\textbf{Proof of Theorem \ref{thm:RegularThm}}]
  According to Lemmas \ref{lem:continuity} and \ref{lem:Contraction}, the map $\mathcal{M}:\mathbb{X}\rightarrow \mathbb{X}$ in \eqref{eq:map} is well-defined in the complete metric space $\mathbb{X}$ and is strictly contractive for $t\in(0,T]$. As a result, it follows from the Banach fixed point theorem \cite[Theorem 5.7]{Brezis10} that there exists a unique fixed point $(n, c, \mathbf{u}) \in \mathbb{X}$ satisfying \eqref{eq:DuhamelP}, which is the unique local mild solution to the system \eqref{eq:problem}-\eqref{eq:InitialC} for $t\in(0,T]$.

 We next prove that $(n,c,\mathbf{u})$ is continuous at $t=0$. To begin with, we have from \eqref{eq:Eg} in Lemma \ref{lem:MLbounds} and H\"{o}lder's inequality that
  \begin{align*}
    \big\|n(t)-n_0\big\|_{L^{q}(\Omega)}
    &\leq \big\|E_{\alpha}(t^{\alpha}\Delta)n_0-n_0\big\|_{L^{q}(\Omega)} \\
    &~~~~+C\int_{0}^{t}(t-s)^{\frac{\alpha}{2}-1}\|n(s)\|_{L^{\rho q}(\Omega)}
     \|\mathbf{u}(s)\|_{L^{\frac{\rho q}{\rho-1}}(\Omega)}ds              \\
    &~~~~+C\int_{0}^{t}(t-s)^{\frac{\alpha}{2}-1}\|n(s)\|_{L^{\rho q}(\Omega)}
     \big\|\nabla c(s)\big\|_{L^{\frac{\rho q}{\rho-1}}(\Omega)}ds         \\
    &\leq \big\|E_{\alpha}(t^{\alpha}\Delta)n_0-n_0\big\|_{L^{q}(\Omega)}
    +Ct^{\frac{\alpha}{2}-2\beta}B(1-2\beta,\frac{\alpha}{2})\|n\|_{\mathcal{X}}
     \|\mathbf{u}\|_{\mathcal{X}}   \\
    &~~~~+Ct^{\frac{\alpha}{2}-2\beta}B(1-2\beta,\frac{\alpha}{2})\|n\|_{\mathcal{X}}
     \big\|\nabla c\big\|_{\mathcal{X}}.
  \end{align*}
  Similarly, by \eqref{eq:Ee} in Lemma \ref{lem:MLbounds} and H\"{o}lder's inequality, it holds that
  \begin{align*}
    &\big\|c(t)-c_0\big\|_{L^{q}(\Omega)}\\
    &\leq \big\|E_{\alpha}(t^{\alpha}(\Delta-\gamma))c_0-c_0\big\|_{L^{q}(\Omega)}
           +C\int_{0}^{t}(t-s)^{\alpha-1}\|\mathbf{u}(s)\|_{L^{\rho q}(\Omega)}
           \big\|\nabla c(s)\big\|_{L^{\frac{\rho q}{\rho-1}}(\Omega)}ds    \\
    &~~~~+C\int_{0}^{t}(t-s)^{\alpha-1}\|n(s)\|_{L^{q}(\Omega)}ds   \\
    &\leq \big\|E_{\alpha}(t^{\alpha}(\Delta-\gamma))c_0-c_0\big\|_{L^{q}(\Omega)}
             +Ct^{\alpha-2\beta}B(1-2\beta,\alpha)\|\mathbf{u}\|_{\mathcal{X}}
             \big\|\nabla c\big\|_{\mathcal{X}}  \\
    &~~~~+Ct^{\alpha-\beta}B(1-\beta,\alpha)\|n\|_{\mathcal{X}},
  \end{align*}
  and
  \begin{align*}
    &\big\|\nabla c(t)-\nabla c_0\big\|_{L^{q}(\Omega)}\\
    &\leq \big\|\nabla E_{\alpha}(t^{\alpha}(\Delta-\gamma))c_0-\nabla c_0\big\|_{L^{q}(\Omega)}
           +Ct^{\frac{\alpha}{2}-2\beta}B(1-2\beta,\frac{\alpha}{2})\|\mathbf{u}\|_{\mathcal{X}}
           \big\|\nabla c\big\|_{\mathcal{X}}  \\
    &~~~~+Ct^{\frac{\alpha}{2}-\beta}B(1-\beta,\frac{\alpha}{2})\|n\|_{\mathcal{X}}.
  \end{align*}
  In addition, it yields from Lemma \ref{lem:Dirichletgroup} and H\"{o}lder's inequality that
  \begin{align*}
   \big\|\mathbf{u}(t)-\mathbf{u}_0\big\|_{L^{q}(\Omega)}
   &\leq  \big\|e^{t\Delta}\mathbf{u}_0-\mathbf{u}_0\big\|_{L^{q}(\Omega)}
       +C\int_{0}^{t}(t-s)^{-\frac{1}{2}-\frac{d}{2q}}\|\mathbf{u}\otimes \mathbf{u}\|_{L^{\frac{q}{2}}(\Omega)}ds   \\
   &~~~~+C\int_{0}^{t}\big\|n\nabla\Phi\big\|_{L^{q}(\Omega)}ds \\
    &\leq  \big\|e^{t\Delta}\mathbf{u}_0-\mathbf{u}_0\big\|_{L^{q}(\Omega)}
         +Ct^{\frac{1}{2}-\frac{d}{2q}-2\beta}
         B(1-2\beta,\frac{1}{2}-\frac{d}{2q})\|\mathbf{u}\|^2_{\mathcal{X}}\\
   &~~~~+Ct^{1-\beta}B(1-\beta,1)\|n\|_{\mathcal{X}}.
  \end{align*}
  Applying Lemma \ref{lem:continuous}, the continuity of $e^{t\Delta}$ and Assumption \ref{ass:pi}, 
  the above estimates tell us that $\lim_{t\rightarrow0^{+}}\|n(t)-n_0\|_{L^{q}(\Omega)}=0$, $\lim_{t\rightarrow0^{+}}\|c(t)-c_0\|_{L^{q}(\Omega)}=0$, $\lim_{t\rightarrow0^{+}}\|\nabla c(t)-\nabla c_0\|_{L^{q}(\Omega)}=0$ and $\lim_{t\rightarrow0^{+}}\|\mathbf{u}(t)-\mathbf{u}_0\|_{L^{q}(\Omega)}=0$. Then we deduce that $n, c\in C([0,T],L^{q}(\Omega))$, and $\nabla c,\mathbf{u}\in C([0,T],L^{q}(\Omega;\mathbb{R}^d))$ due to $L^{\rho q}(\Omega)\hookrightarrow L^{q}(\Omega)$. In addition, the asymptotic property in \eqref{eq:asymptotic-behav} is also obtained from the above estimates.

  It remains to show the continuous dependence on the initial data, which is crucial for inferring the continuation of the mild solution to the TF-KSNS system in the next section.
  Let $(n_1,c_1,\mathbf{u}_1)$ and $(n_2,c_2,\mathbf{u}_2)$ be the mild solutions in $\mathbb{X}$ to the TF-KSNS system \eqref{eq:problem}-\eqref{eq:InitialC} corresponding to the initial data $(n_{1,0},c_{1,0},\mathbf{u}_{1,0})$ and $(n_{2,0},c_{2,0},\mathbf{u}_{2,0})$, respectively. Since $\|\nabla E_{\alpha}(t^{\alpha}(\Delta-\gamma))(c_{1,0}-c_{2,0})\|_{L^{\rho q}(\Omega)}\leq C\|\nabla (c_{1,0}-c_{2,0})\|_{L^{\rho q}(\Omega)}$ holds by applying \eqref{eq:HG1} in Lemma \ref{lem:Heatgroup} and \eqref{eq:Mainardi}, then for $T>0$ being small enough, it yields from \eqref{eq:DuhamelP} and \eqref{eq:map} that
  \begin{align*}
    &\big\|n_1-n_2\big\|_{\mathcal{X}}+\big\|c_1-c_2\big\|_{\mathcal{X}}+\big\|\nabla(c_1-c_2)\big\|_{\mathcal{X}}
       +\big\|\mathbf{u}_1-\mathbf{u}_2\big\|_{\mathcal{X}} \\
    &\leq  C\big\|n_{1,0}-n_{2,0}\big\|_{\mathcal{X}}+C\big\|c_{1,0}-c_{2,0}\big\|_{\mathcal{X}}
       +C\big\|\nabla(c_{1,0}-c_{2,0})\big\|_{\mathcal{X}}+C\big\|\mathbf{u}_{1,0}-\mathbf{u}_{2,0}\|_{\mathcal{X}}\\
    &~~+\big\|\mathcal{M}_1(n_1, c_1, \mathbf{u}_1)-\mathcal{M}_1(n_2, c_2,\mathbf{u}_2)\big\|_{\mathcal{X}}
       +\big\|\mathcal{M}_2(n_1, c_1,\mathbf{u}_1)-\mathcal{M}_2(n_2, c_2, \mathbf{u}_2)\big\|_{\mathcal{X}}\\
    &~~+\big\|\nabla\mathcal{M}_2(n_1, c_1,\mathbf{u}_1)-\nabla\mathcal{M}_2(n_2, c_2,
       \mathbf{u}_2)\big\|_{\mathcal{X}}+\big\|\mathcal{M}_3(n_1, c_1,
       \mathbf{u}_1)-\mathcal{M}_3(n_2, c_2, \mathbf{u}_2)\big\|_{\mathcal{X}},
  \end{align*}
  Hence, by \eqref{eq:M1-nu}, \eqref{eq:M2-nu}, \eqref{eq:M2-gnu} and \eqref{eq:M3-nu}, it is easy to obtain that
  \begin{align*}
    &\|n_1-n_2\|_{\mathcal{X}}+\|c_1-c_2\|_{\mathcal{X}}+\big\|\nabla(c_1-c_2)
     \big\|_{\mathcal{X}}+\|\mathbf{u}_1-\mathbf{u}_2\|_{\mathcal{X}}\\
    &\leq C(\|n_{1,0}-n_{2,0}\|_{\mathcal{X}}+\|c_{1,0}-c_{2,0}\|_{\mathcal{X}}
     +\big\|\nabla(c_{1,0}-c_{2,0})\big\|_{\mathcal{X}}+\|\mathbf{u}_{1,0}-\mathbf{u}_{2,0}\|_{\mathcal{X}}).
  \end{align*}
  Therefore, the proof of Theorem \ref{thm:RegularThm} is completed.
\end{proof}

\section{Blow-up of the mild solution}
\label{sec:blow-up}
Starting from a smooth initial configuration and after a first period of classical evolution, the phenomenon in which the solution (or in some cases its derivatives) becomes infinite in finite time due to the cumulative effect of the nonlinearities is called blow-up (e.g., \cite{Quittner07}). We shall discuss the blow-up of the solution to the TF-KSNS system (\ref{eq:problem})-(\ref{eq:InitialC}). To this end, according to \cite{Costa23,Andrade22}, we first introduce the following definition of the continuation of the solution.

\begin{definition}[Continuation of the solution]
  Let $(n,c,\mathbf{u}):[0,T]\rightarrow L^{q}(\Omega;\mathbb{R}^{2+d})$ be a mild solution to the TF-KSNS system \eqref{eq:problem}-\eqref{eq:InitialC}. For $\bar{T}>T$, we say that $(\bar{n},\bar{c},\bar{\mathbf{u}}):[0,\bar{T}]\rightarrow L^{q}(\Omega;\mathbb{R}^{2+d})$ is a continuation of $(n,c,\mathbf{u})$ if $(\bar{n},\bar{c},\bar{\mathbf{u}})$ is a mild solution and $n(t)\equiv\bar{n}(t)$, $c(t)\equiv\bar{c}(t)$, $\mathbf{u}(t)\equiv\bar{\mathbf{u}}(t)$ as $t\in[0,T]$.
\end{definition}

Combining the strategies in \cite{Costa23,Andrade22,Jiang22,Li2018}, the proof of Theorem \ref{thm:Blow-up} consists of two parts: (i) the mild solution can be continuously extended, and such a continuation is unique; (ii) the assertion on $T_{max}$ in Theorem \ref{thm:Blow-up} holds by the contradiction method.

For these purposes, we introduce a complete metric space $\bar{\mathbb{S}}_{T}$ as follows. Let $(n,c,\mathbf{u}):[0,T]\rightarrow L^{\rho q}(\Omega;\mathbb{R}^{2+d})$ be the mild solution in $\mathbb{X}$ to the TF-KSNS system \eqref{eq:problem}-\eqref{eq:InitialC}, and denote $\bar{\mathcal{X}}:=\{v(t)\in C((0,\bar{T}],L^{\rho q}(\Omega)):\sup_{t\in(0,\bar{T})}t^{\beta}\|v(t)\|_{L^{\rho q}(\Omega)}<+\infty\}$, which is a Banach space endowed with the norm $\|\bar{v}\|_{\bar{\mathcal{X}}}=\sup_{t\in(0,\bar{T})}t^{\beta}\|\bar{v}(t)\|_{L^{\rho q}(\Omega)}$. Taking $\bar{T}<\infty$ close to $T$ and satisfying $\bar{T}>T$, we extend the map $\mathcal{M}$ in \eqref{eq:map} to the complete metric space $\bar{\mathbb{S}}_{T}$, and
\begin{align*}
  \bar{\mathbb{S}}_{T}
  :=\left\{\small(\bar{n},\bar{c},\bar{\mathbf{u}})\left|~
  \begin{aligned}
    &\bar{n}\in \bar{\mathcal{X}},~\bar{c}\in
       \bar{\mathcal{X}},~\nabla\bar{c}\in\bar{\mathcal{X}}^d,~\bar{\mathbf{u}}\in \bar{\mathcal{X}}^d, \\
    &\sup_{t\in[T,\bar{T}]}\|\bar{n}(t)-n(T)\|_{L^{\rho q}(\Omega)}
      +\sup_{t\in[T,\bar{T}]}\|\bar{c}(t)-c(T)\|_{L^{\rho q}(\Omega)}\\
    &+\sup_{t\in[T,\bar{T}]}\|\nabla \bar{c}(t)-\nabla c(T)\|_{L^{\rho q}(\Omega)}
      +\sup_{t\in[T,\bar{T}]}\|\bar{\mathbf{u}}(t)-\mathbf{u}(T)\|_{L^{\rho q}(\Omega)}\le R,\\
    &\bar{n}\equiv n,~\bar{c}\equiv c,~\bar{\mathbf{u}}\equiv \mathbf{u}~ {\rm in}~ [0,T]
  \end{aligned}\right.\right\},
\end{align*}
which is endowed with the metric $\mathcal{D}_{\bar{T}}$ given by
\begin{displaymath}
  \mathcal{D}_{\bar{T}}[(\bar{n}_1,\bar{c}_1,\bar{\mathbf{u}}_1),(\bar{n}_2,\bar{c}_2,\bar{\mathbf{u}}_2)]
  :=\|\bar{n}_1-\bar{n}_2\|_{\bar{\mathcal{X}}}
   +\|\bar{c}_1-\bar{c}_2\|_{\bar{\mathcal{X}}}
   +\|\nabla(\bar{c}_1-\bar{c}_2)\|_{\bar{\mathcal{X}}}
   +\|\bar{\mathbf{u}}_1-\bar{\mathbf{u}}_2\|_{\bar{\mathcal{X}}}.
\end{displaymath}

\begin{lemma}\label{lem:Ex-continuous}
  For $\bar{T}>T$ close enough to $T$, the operator $\mathcal{M}$ maps $\bar{\mathbb{S}}_{T}$ to itself, and $\mathcal{M}(\bar{n},\bar{c},\bar{\mathbf{u}})(t)$ is continuous with respect to $t\in[0,\bar{T}]$ for $(\bar{n},\bar{c},\bar{\mathbf{u}})\in\bar{\mathbb{S}}_{T}$.
\end{lemma}
\begin{proof}
  Let $(\bar{n},\bar{c},\bar{\mathbf{u}})\in\bar{\mathbb{S}}_{T}$. It indicates that $\bar{n}\equiv n$, $\bar{c}\equiv c$ and $\bar{\mathbf{u}}\equiv \mathbf{u}$ for $t\in[0,T]$. Then the continuity of $\mathcal{M}(\bar{n},\bar{c},\bar{\mathbf{u}})(t)$ in the interval $[0,T]$ is guaranteed by Theorem \ref{thm:RegularThm}. By the similar approach as in the proof of Lemma \ref{lem:cMt}, we can obtain that $\mathcal{M}(\bar{n},\bar{c},\bar{\mathbf{u}})(t)$ is also continuous for $t\in(T,\bar{T}]$.

  Next, we shall show that $\mathcal{M}(\bar{n},\bar{c},\bar{\mathbf{u}})(t)$ is continuous at $t=T$.  Let $T<t<\bar{T}$, it follows from \eqref{eq:map} and \eqref{eq:DuhamelP} that
  \begin{align*}
    &\mathcal{M}_1(\bar{n},\bar{c},\bar{\mathbf{u}})(t)-n(T)\\
    &=(E_{\alpha}(t^{\alpha}\Delta)-E_{\alpha}(T^{\alpha}\Delta))n_0-\int_{T}^{t}(t-s)^{\alpha-1}E_{\alpha,\alpha}\big((t-s)^{\alpha}\Delta\big)
             \nabla\cdot\left(\bar{\mathbf{u}}\bar{n}+\bar{n}\nabla \bar{c}\right)(s)ds\\
    &-\int_{0}^{T}\big[(t-s)^{\alpha-1}E_{\alpha,\alpha}\big((t-s)^{\alpha}\Delta\big)
             -(T-s)^{\alpha-1}E_{\alpha,\alpha}\big((T-s)^{\alpha}\Delta\big)\big]
             \nabla\cdot\big(\mathbf{u}n+n\nabla c\big)(s)ds\\
    &:= \mathcal{J}^{\mathcal{M}_1}_{1}(t)+ \mathcal{J}^{\mathcal{M}_1}_{2}(t)
           +\mathcal{J}^{\mathcal{M}_1}_{3}(t).
  \end{align*}
  By Lemma \ref{lem:continuous}, we have that $\mathcal{J}^{\mathcal{M}_1}_{1}(t)$ goes to zero in $L^{\rho q}(\Omega)$ as $t\rightarrow T^{+}$. Analogous to the estimates for $I^{\mathcal{M}_1}_{4}$ and $I^{\mathcal{M}_1}_{6}$ in Lemma \ref{lem:cMt}, it obtains by the Lebesgue dominated convergence theorem that $\mathcal{J}^{\mathcal{M}_1}_{3}(t)$ tends to zero as $t\rightarrow T^{+}$. Regarding $\mathcal{J}^{\mathcal{M}_1}_{2}(t)$, by the similar approach for \eref{eq:I1} in Lemma \ref{lem:MLbounds}, we can infer that
  \begin{align*}
    &\big\|\mathcal{J}^{\mathcal{M}_1}_{2}(t)\big\|_{L^{\rho q}(\Omega)}\\
    &\leq C\int_{T}^{t}(t-s)^{\alpha-1}\big(1+(t-s)^{-\frac{\alpha}{2}-\frac{\alpha d}{2q}}\big)
             \big[\|\bar{\mathbf{u}}\bar{n}(s)\|_{L^{\frac{\rho q}{\rho+1}}(\Omega)}+\|\bar{n}\nabla \bar{c}(s)\|_{\frac{\rho q}{\rho+1}(\Omega)}\big]ds\\
    &\leq C\int_{T}^{t}(t-s)^{\frac{\alpha}{2}-\frac{\alpha d}{2q}-1}s^{-2\beta}
            \|\bar{n}(s)\|_{L^{\rho q}(\Omega)}
           \left[\|\bar{\mathbf{u}}(s)\|_{L^{\rho q}(\Omega)}+\|\nabla \bar{c}(s)\|_{L^{\rho q}(\Omega)}\right]ds\\
    &\leq C_1\int_{T/t}^{1}(1-\tau)^{\frac{\alpha}{2}-\frac{\alpha d}{2q}-1}\tau^{-2\beta}ds,
  \end{align*}
  where $C_1=C\big(R+\|n(T)\|_{L^{\rho q}(\Omega)}\big)\big(2R+\|\mathbf{u}(T)\|_{L^{\rho q}(\Omega)}+\|\nabla c(T)\|_{L^{\rho q}(\Omega)}\big)$, then it shows that $\|\mathcal{J}^{\mathcal{M}_1}_{2}(t)\|_{L^{\rho q}(\Omega)}\rightarrow0$ as $t\rightarrow T^{+}$.
  Consequently, we can take $\bar{T}$ close enough to $T^{+}$ such that
  \begin{equation}
    \sup_{t\in[T,\bar{T}]}\|\mathcal{M}_1(\bar{n},\bar{c},\bar{\mathbf{u}})(t)-n(T)\|_{L^{\rho q}(\Omega)}\leq\frac{R}{4}.
  \end{equation}
  By analogy, we can assert that $\mathcal{M}_2(\bar{n}$, $\bar{c},\bar{\mathbf{u}})(t)$, $\nabla\mathcal{M}_2(\bar{n},\bar{c},\bar{\mathbf{u}})(t)$ and $\mathcal{M}_3(\bar{n},\bar{c},\bar{\mathbf{u}})(t)$ are continuous over $[0,\bar{T}]$, then it also holds that $\sup_{t\in[T,\bar{T}]}\|\mathcal{M}_2(\bar{n},\bar{c}$, $\bar{\mathbf{u}})(t)-c(T)\|_{L^{\rho q}(\Omega)}\leq R/4$, $\sup_{t\in[T,\bar{T}]}\|\nabla\mathcal{M}_2(\bar{n},\bar{c},\bar{\mathbf{u}})(t)-\nabla c(T)\|_{L^{\rho q}(\Omega)}\leq R/4$, and $\sup_{t\in[T,\bar{T}]}\|\mathcal{M}_3(\bar{n},\bar{c},\bar{\mathbf{u}})(t)-\mathbf{u}(T)\|_{L^{\rho q}(\Omega)}\leq R/4$ for $\bar{T}$ close enough to $T$. Hence, we have $\mathcal{M}(\bar{n},\bar{c},\bar{\mathbf{u}})\in\bar{\mathbb{S}}_{T}$.
\end{proof}

\begin{lemma}\label{lem:Ex-contraction}
  If $(\bar{n},\bar{c},\bar{\mathbf{u}})\in\bar{\mathbb{S}}_{T}$, then the map $\mathcal{M}$ is a contraction, and there exists a unique solution $(\bar{n},\bar{c},\bar{\mathbf{u}})$ being the continuation of the mild solution $(n,c,\mathbf{u})$ to the interval $[0,\bar{T}]$.
\end{lemma}
\begin{proof}
  The analysis is analogous to that of Lemma \ref{lem:Contraction}.
  Let $(\bar{n}_1,\bar{c}_1,\bar{\mathbf{u}}_1)$ and $(\bar{n}_2,\bar{c}_2,\bar{\mathbf{u}}_2)$ belong to $\bar{\mathbb{S}}_{T}$. We first consider to estimate $\|\mathcal{M}_1(\bar{n}_1,\bar{c}_1,\bar{\mathbf{u}}_1)-\mathcal{M}_1(\bar{n}_2,\bar{c}_2,\bar{\mathbf{u}}_2)\|_{\bar{\mathcal{X}}}$. It consists of analyzing the terms $J^{\mathcal{M}_1}_{j}$ ($j=1,\cdots,4$) as in Lemma \ref{lem:Contraction} with $T$ replaced by $\bar{T}$. By the similar approach, we have
  \begin{align*}
    &\|\mathcal{M}_1(\bar{n}_1,\bar{c}_1,\bar{\mathbf{u}}_1)-\mathcal{M}_1(\bar{n}_2,\bar{c}_2,\bar{\mathbf{u}}_2)\|_{\bar{\mathcal{X}}}\\
    &\leq C_1\max\big(3R/8,\bar{T}^{\beta}R/4+\bar{T}^{\beta}\|n_2(T)\|_{L^{\rho q}(\Omega)}\big)
          \|\mathbf{u}_1-\mathbf{u}_2\|_{\bar{\mathcal{X}}}\\
    &~~~~+C_1\max\big(3R/8,\bar{T}^{\beta}R/4+\bar{T}^{\beta}\|\mathbf{u}_1(T)\|_{L^{\rho q}(\Omega)}\big)\|n_1-n_2\|_{\bar{\mathcal{X}}}\\
    &~~~~+C_1\max\big(3R/8,\bar{T}^{\beta}R/4+\bar{T}^{\beta}\|\nabla c_2(T)\|_{L^{\rho q}(\Omega)}\big)\|n_1-n_2\|_{\bar{\mathcal{X}}}\\
    &~~~~+C_1\max\big(3R/8,\bar{T}^{\beta}R/4+\bar{T}^{\beta}\|n_1(T)\|_{L^{\rho q}(\Omega)}\big)\|\nabla c_1-\nabla c_2\|_{\bar{\mathcal{X}}},
  \end{align*}
  where $C_1=C\bar{T}^{\frac{\alpha}{2}-\frac{\alpha d}{2q}-\beta}B(1-2\beta,\frac{\alpha}{2}-\frac{\alpha d}{2q})$. For $\mathcal{M}_2$, we can obtain that
  \begin{align*}
    &\|\mathcal{M}_2(\bar{n}_1,\bar{c}_1,\bar{\mathbf{u}}_1)-\mathcal{M}_2(\bar{n}_2,\bar{c}_2,\bar{\mathbf{u}}_2)\|_{\bar{\mathcal{X}}}\\
    &\leq C_2\max\big(3R/8,\bar{T}^{\beta}R/4+\bar{T}^{\beta}\|\nabla c_1(T)\|_{L^{\rho q}(\Omega)}\big)\|\bar{\mathbf{u}}_1-\bar{\mathbf{u}}_2\|_{\bar{\mathcal{X}}}\\
    &~~~~+C_2\max\big(3R/8,\bar{T}^{\beta}R/4+\bar{T}^{\beta}\|\mathbf{u}_2(T)\|_{L^{\rho q}(\Omega)}\big)\|\nabla c_1-\nabla c_2\|_{\bar{\mathcal{X}}}\\
    &~~~~+C\bar{T}^{\alpha}B(1-\beta,\alpha)\|n_1-n_2\|_{\bar{\mathcal{X}}},
  \end{align*}
  in which $C_2=C\bar{T}^{\alpha-\frac{\alpha d}{2q}-\beta}B(1-2\beta,\alpha-\frac{\alpha d}{2q})$.
  Regarding the estimate for $\nabla\mathcal{M}_2$, it holds that
  \begin{align*}
    &\|\nabla\mathcal{M}_2(\bar{n}_1,\bar{c}_1,\bar{\mathbf{u}}_1)-\nabla\mathcal{M}_2(\bar{n}_2,\bar{c}_2,\bar{\mathbf{u}}_2)\|_{\bar{\mathcal{X}}}\\
    &\leq C_1\max\big(3R/8,\bar{T}^{\beta}R/4+\bar{T}^{\beta}\|\nabla c_1(T)\|_{L^{\rho q}(\Omega)}\big)\|\mathbf{u}_1-\mathbf{u}_2\|_{\bar{\mathcal{X}}},\\
    &~~~~+C_1\max\big(3R/8,\bar{T}^{\beta}R/4+\bar{T}^{\beta}\|\mathbf{u}_2(T)\|_{L^{\rho q}(\Omega)}\big)\|\nabla c_1-\nabla c_2\|_{\bar{\mathcal{X}}}\\
    &~~~~+C\bar{T}^{\frac{\alpha}{2}}B(1-\beta,\frac{\alpha}{2})\|n_1-n_2\|_{\bar{\mathcal{X}}}.
  \end{align*}
  In addition, let $C_3=C\bar{T}^{\frac{1}{2}-\frac{d}{2\rho q}-\beta}B(1-2\beta, \frac{1}{2}-\frac{d}{2\rho q})$, the estimate for $\mathcal{M}_3$ is as follows
  \begin{align*}
    &\|\mathcal{M}_3(\bar{n}_1,\bar{c}_1,\bar{\mathbf{u}}_1)-\mathcal{M}_3(\bar{n}_2,\bar{c}_2,\bar{\mathbf{u}}_2)\|_{\bar{\mathcal{X}}}\\
    &\leq C_3\max\big(3R/8,\bar{T}^{\beta}R/4+\bar{T}^{\beta}\|\mathbf{u}_1(T)\|_{L^{\rho q}(\Omega)}\big)\|\mathbf{u}_1-\mathbf{u}_2\|_{\bar{\mathcal{X}}}\\
    &~~~~+C_3\max\big(3R/8,\bar{T}^{\beta}R/4+\bar{T}^{\beta}\|\mathbf{u}_2(T)\|_{L^{\rho q}(\Omega)}\big)\|\mathbf{u}_1-\mathbf{u}_2\|_{\bar{\mathcal{X}}}\\
    &~~~~+C\bar{T}\|n_1-n_2\|_{\mathcal{X}}.
  \end{align*}
  Then, by choosing an appropriate $\bar{T}$, the contraction of $\mathcal{M}$ can be guaranteed
  \begin{align*}
    \mathcal{D}_{\bar{T}}[\mathcal{M}(\bar{n}_1,\bar{c}_1,\bar{\mathbf{u}}_1),\mathcal{M}(\bar{n}_2, \bar{c}_2, \bar{\mathbf{u}}_2)]\leq\bar{\vartheta} \mathcal{D}_{\bar{T}}[(\bar{n}_1,\bar{c}_1,\bar{\mathbf{u}}_1),(\bar{n}_2,\bar{c}_2,\bar{\mathbf{u}}_2)],~~ \text{with}~~ 0<\bar{\vartheta}<1.
  \end{align*}
  Consequently, such a continuation is uniquely established by the Banach fixed point theorem.
\end{proof}

Based on the above lemmas, a detailed discussion on the blow-up of the mild solution to the TF-KSNS system \eqref{eq:problem}-\eqref{eq:InitialC} is provided as follows for Theorem \ref{thm:Blow-up}.

\begin{proof}[\textbf{Proof of Theorem \ref{thm:Blow-up}}]
  We prove the result using a contradiction approach.
  By Lemma \ref{lem:Ex-contraction}, we suppose that the mild solution $(n,c,\mathbf{u})$ is uniquely continued up to a maximal time $T_{max}<+\infty$, and for any $t\in[0,T_{max})$,
  \begin{align*}
    \|n(t)\|_{L^{\rho q}(\Omega)}<\infty,~ ~
    \|c(t)\|_{L^{\rho q}(\Omega)}<\infty,~~{\rm and}~~
    \|\mathbf{u}(t)\|_{L^{\rho q}(\Omega)}<\infty.
  \end{align*}
  Let $\{t_k\}_{k=0}^{\infty}$ be a sequence in $[0,T_{max})$ and $t_k\rightarrow T^{-}_{max}$ as $k\rightarrow+\infty$. For any $m,k\in\mathbb{N}$ with $0<t_m<t_k<T_{max}$, it has $t_m,t_k\rightarrow T^{-}_{max}$ as $m,k\rightarrow+\infty$. Then, by the analogous estimates in the proof of Lemma \ref{lem:cMt}, for $m$, $k\rightarrow+\infty$, we have
  \begin{align*}
    \|n(t_m)-n(t_k)\|_{L^{\rho q}(\Omega)}+\|c(t_m)-c(t_k)\|_{L^{\rho q}(\Omega)}
    +\|\mathbf{u}(t_m)-\mathbf{u}(t_k)\|_{L^{\rho q}(\Omega)}\rightarrow0,
  \end{align*}
  which indicates that $\{(n(t_k),c(t_k),\mathbf{u}(t_k))\}$ is a Cauchy sequence and convergent as $t_k\rightarrow T^{-}_{max}$. Hence, $n(T_{max})$, $c(T_{max})$ and $\mathbf{u}(T_{max})$ can be defined as the limit, and $(n,c,\mathbf{u})$ can further continued beyond $T_{max}$, which leads to a contradiction.
\end{proof}
\section{Conclusions}
\label{sec:Conclusions}
In this paper, we consider the mathematical modeling and analysis for the chemotactic diffusion kinetics of myxobacteria in the soil (porous medium) under the influence of liquid flow fields (non-static environment).
In such a specific biological process, we first derive the time-fractional Keller-Segel system for describing the chemotactic diffusion of myxobacteria and slime from the CTRW approach, and then couple it with the Navier-Stokes equations for describing the influence of liquid flow field to obtain the TF-KSNS system (\ref{eq:problem}). To further investigate the mathematical properties of the system with appropriate initial and boundary conditions, we construct a suitable metric space and apply the Banach fixed point theorem to obtain the local well-posedness, dependence on initial values, and asymptotic properties of the mild solution. The continuity and uniqueness of the continuation are also proved, as well as the blow-up of the solution. The discussions lay a foundation for further investigations, such as the existence and uniqueness of global solutions under specific initial conditions and the solvability in Besov spaces. In addition, numerical approximation, simulation, and numerical analysis for the system are left in future works.

\section*{Declaration of competing interest}
There is no conflict of interests.


\section*{Acknowledgments}
\label{sec:Acknowledgments}
This work was supported by the National Natural Science Foundation of China under Grant Nos. 12225107 and 12071195, the Major Science and Technology Projects in Gansu Province-Leading Talents in Science and Technology under Grant No. 23ZDKA0005, the Innovative Groups of Basic Research in Gansu Province under Grant No. 22JR5RA391, and Lanzhou Talent Work Special Fund.

\section*{References}


\end{document}